\documentclass[11pt,a4paper,reqno]{amsart}
\usepackage{amsmath,amssymb,graphics,epsfig,color,enumerate}
\usepackage{dsfont}  
\usepackage{verbatim}
\usepackage{epstopdf}
\usepackage[foot]{amsaddr}

 \definecolor{refkey}{gray}{0.8}
 \definecolor{labelkey}{gray}{0.8}
\newtheorem{Theorem}{Theorem}[section]
\newtheorem{TheoremA}{Theorem}
\newtheorem{CorollaryA}{Corollary}
\newtheorem{Lemma}[Theorem]{Lemma}
\newtheorem{Proposition}[Theorem]{Proposition}

\newtheorem{remark}[Theorem]{Remark}

\newtheorem{claim}{Claim}
\newtheorem{definition}[Theorem]{Definition}

\newtheorem{assumptionA}{Assumption}

 \definecolor{darkgreen}{rgb}{0,0.7,0}
\definecolor{yac}{rgb}{0.0, 0.44, 1.0}

\definecolor{light}{gray}{.9}

\newcommand{\verde}{\textcolor{black}}

\newcommand{\cB}{\ensuremath{\mathcal B}}
\newcommand{\cC}{\ensuremath{\mathcal C}}
\newcommand{\cD}{\ensuremath{\mathcal D}}
\newcommand{\cE}{\ensuremath{\mathcal E}}
\newcommand{\cF}{\ensuremath{\mathcal F}}
\newcommand{\cG}{\ensuremath{\mathcal G}}
\newcommand{\cH}{\ensuremath{\mathcal H}}

\newcommand{\cL}{\ensuremath{\mathcal L}}
\newcommand{\cM}{\ensuremath{\mathcal M}}
\newcommand{\cN}{\ensuremath{\mathcal N}}
\newcommand{\cO}{\ensuremath{\mathcal O}}
\newcommand{\cP}{\ensuremath{\mathcal P}}

\newcommand{\cS}{\ensuremath{\mathcal S}}
\newcommand{\cT}{\ensuremath{\mathcal T}}
\newcommand{\cU}{\ensuremath{\mathcal U}}

\newcommand{\bbE}{{\ensuremath{\mathbb E}} }

\newcommand{\bbL}{{\ensuremath{\mathbb L}} }

\newcommand{\bbN}{{\ensuremath{\mathbb N}} }

\newcommand{\bbP}{{\ensuremath{\mathbb P}} }

\newcommand{\bbR}{{\ensuremath{\mathbb R}} }

\newcommand{\bbZ}{{\ensuremath{\mathbb Z}} }

\newcommand{\var}{\operatorname{Var}}

\let\a=\alpha \let\b=\beta   \let\d=\delta  \let\e=\varepsilon
 \let\g=\gamma     \let\k=\kappa  \let\l=\lambda
            
  \let\s=\sigma \let\t=\tau   
  \let\z=\zeta
\let\D=\Delta   \let\G=\Gamma  \let\L=\Lambda 
\let\O=\Omega      

\begin{document}

\title[Random walks in dynamic random environments via $L^2$--perturbations]{Analysis of random walks in dynamic random environments via $L^2$--perturbations}

\author{L. Avena$^1$}
\address{$^1$Mathematisch instituut
Universiteit Leiden. Postbus 9512
2300 RA Leiden,
The Netherlands. Supported by NWO Gravitation Grant 024.002.003-NETWORKS}
\email{l.avena@math.leidenuniv.nl}
\author{O. Blondel$^2$}
\address{$^2$CNRS, Univ Lyon, Universit\'e Claude Bernard Lyon 1, ICJ, CNRS UMR 5208; 
43 blvd. du 11 novembre 1918
F-69622 Villeurbanne cedex,
France}
\email{blondel@math.univ-lyon1.fr}
\author{A. Faggionato$^3$}
\address{$^3$Dipartimento di Matematica, Universit\`a di Roma La Sapienza.
  P.le Aldo Moro 2, 00185 Roma, Italy}
\email{faggiona@mat.uniroma1.it}

\begin{abstract} 

{We consider random walks in dynamic random environments given by Markovian dynamics {on $\mathbb{Z}^d$}.{ We assume that the environment has a stationary distribution $\mu$ and
satisfies the Poincar\'e inequality w.r.t. $\mu$. The random walk is a perturbation of another  random walk (called
``unperturbed''). We assume that also the environment viewed from the
unperturbed random  walk has stationary distribution $\mu$.} Both perturbed and unperturbed random walks can depend heavily on the environment {and are not assumed to be finite--range}. We derive a law of large numbers, an averaged invariance principle for the position of the walker and a series expansion for the asymptotic speed. We also provide a condition for non-degeneracy of the diffusion, and describe in some details equilibrium and convergence properties of the environment seen by the walker.}
{All these results are based on a more general perturbative analysis of operators that we derive in the context of {$L^2$}--bounded perturbations of Markov processes by means of the so--called Dyson--Phillips expansion.}


\smallskip

\end{abstract}
\maketitle



\section{Introduction} 
{ Random motion in random media has been the subject of intensive studies in the physics and mathematics literature over  the last decades. The main motivation to our work  is the analysis of    rather general continuous--time Random Walks (RWs) on $\bbZ^d$, whose transition rates are given as a function of an underlying (autonomous) Markov process playing the role of a dynamic random environment.


 A number of results (as LLN, CLT, large deviation estimates) have been obtained in the past under various conditions that allow some control on the strong dependence between the trajectories  of the random walk and the environment. We mention space and/or time independence assumptions on the environment (see e.g. \cite{BMP} for quenched CLT of perturbation of simple random walks using cluster expansion, \cite{BK} for diffusive bounds by using renormalization techniques, \cite{RAS2} for quenched invariance principles by analyzing the environment as seen by the walk, \cite{BZ} for a law of large numbers and a high-dimensional quenched invariance principle by constructing regeneration times) and  balanced conditions (cf. \cite{DFGW} for averaged invariance principles under reversibility of the environment as seen by the walker and \cite{DGR} for a quenched invariance principle for balanced random walks). 
When allowing non--trivial space-time correlation structures, in \cite{AdHR}
 for some   uniformly elliptic walks and in  \cite{HSS} for   non--elliptic ones, laws of large numbers via regeneration times have been established by assuming mixing conditions on the environment that are uniform on the initial configuration (i.e. adaptation of \emph{cone-mixing} conditions borrowed from \cite{CZ} for static random environments).
In a similar setting, a quenched CLT has been established in \cite{DKL}, and a quite general asymptotic analysis has been pursued in the recent \cite{RV}, again by using a uniform mixing condition expressed in terms of a coupling.  When dealing with poorly--mixing environments, some progress has been recently achieved by using highly model dependent techniques {\cite{ASV,H+,HS15}}.

In this work, we require that the environment satisfies an \emph{exponential $L^2$--mixing} hypothesis (namely, the Poincar\'e inequality w.r.t. an invariant distribution $\mu$) and that the random walk is ``close to nice'', in the sense that it is a perturbation of a random walk  such that  $\mu$ is an invariant distribution for the environment viewed by the walker.
 We stress that even though we are in a perturbative setting, the reference unperturbed random walk is allowed to depend strongly on the environment. Moreover, unlike most of the references above, we do not require finite range for the jumps of the walk. 
{As discussed in  Section \ref{applications}},   we establish several results for the RW and for the environment seen from it. For the latter, we show that there exists a unique invariant distribution absolutely continuous w.r.t. $\mu$, 
we analyze convergence to this  invariant measure and  ergodicity, we derive  an expansion of its density w.r.t.\@ $\mu$ and show that the effect of the perturbation on  the  density is sharply localized  around the origin, and we derive an exponential $L^2$--mixing property similar to the Poincar\'e inequality  (see Theorems~\ref{late},~\ref{porte-bonheur},~\ref{canicule}). For the random walk itself, we prove a LLN and an averaged invariance principle, as well as the non-degeneracy of the diffusion matrix under suitable conditions (see Theorems~\ref{late},~\ref{QCLT}).

One of the  basic tools for the above  results is the so--called Dyson--Phillips expansion, which we use to derive a series expansion for the semigroup of the environment seen from the walker. This perturbative analysis is very general, and indeed {in Section \ref{viva_segovia}}  it is carried on for a generic Markov process stationary w.r.t. some invariant and ergodic  distribution $\mu$ and satisfying the Poincar\'e inequality. We assume that  the generator of
the  perturbed Markov process is  (roughly speaking)  obtained by a $L^2(\mu)$--bounded   perturbation  of the generator of the original, unperturbed, Markov process. In Theorem \ref{teo_invariante} we prove that the perturbed process admits a unique invariant distribution absolutely continuous w.r.t. $\mu$ (which is also ergodic), write a series expansion for its density w.r.t. $\mu$ and for the perturbed  Markov semigroup,  
and estimate the convergence to equilibrium for the latter. 
In addition, in  Corollary \ref{lacho_drom} and Proposition \ref{gelem_gelem}, we state  a law of large numbers and an invariance principle for additive functionals of the perturbed Markov process, respectively.

Let us further comment on some   closely related works with perturbative techniques. In \cite{AdHR} the Dyson--Phillips expansion has also been used in a similar fashion in one of the main results therein, but the authors only focus on the law of large numbers for the walk 	and work  under the more restrictive sup--norm instead of the $L^2$--norm. In \cite{KO} the authors work with hypotheses very similar to our own for Theorem \ref{teo_invariante} (even allowing more general perturbations), but the obtained results present some differences.  In particular, in \cite{KO} the uniqueness of the invariant distribution for the perturbed process is proved inside the 
smaller class of distributions whose  density w.r.t. $\mu$ is bounded in $L^2(\mu)$. In  addition, in Theorem \ref{teo_invariante} we derive  information on the exponential  convergence  of the perturbed  semigroup (which is relevant to get the invariance principle in  Proposition \ref{gelem_gelem}), while in \cite{KO}  the exponential  convergence of the perturbed densities is derived. For more detailed comments on the relation between \cite{KO} and our Theorem \ref{teo_invariante}  {we refer to 
  Remark \ref{KOmark}}.
  We point out that the  main goal  in \cite{KO}   is to establish the Einstein relation for the speed of the walker, hence we have not focused on this issue since already treated there.  Finally,  we mention  the recent work \cite{ramirez}, where the author considers perturbations of infinite dimensional diffusions with known invariant measure (not necessarily reversible), satisfying the log-Sobolev inequality (which is stronger than the Poincar\'e inequality).  The invariant measure for the perturbed process is analyzed and its density is  expressed in terms of a series expansion similar to \eqref{muinfinity_bis}, \eqref{expansiong} below.

Finally, we mention that the results we present herein can be pushed to obtain more detailed information when dealing with explicit examples of random walks in dynamic random environments. This path has been pursued in \cite{ABF}, where we consider one-dimensional examples in which the dynamic environments are given by kinetically constrained models.

\medskip

\noindent
{{\bf Outline of the paper}}.
 {In Section \ref{applications} we present our main results concerning random walks in dynamic random environments, i.e. 
Theorems \ref{late},  \ref{porte-bonheur}, \ref{canicule} and \ref{QCLT}. 
The main results concerning perturbations of {more general} Markov processes, i.e. Theorem \ref{teo_invariante},  Corollary \ref{lacho_drom} and Proposition \ref{gelem_gelem}, are stated in Section \ref{viva_segovia}. The other sections, from  \ref{golden_gate}  to \ref{NonDeg}, are devoted to the proofs of the above statements. 
{In particular, in Section \ref{san_valentino} we
present a coupling construction allowing to compare perturbed and unperturbed walkers which is independent of the small
perturbation assumption.}
Finally, in Appendix \ref{app_misc} we derive some simple but useful analytic results}.




\section{Random walks  in dynamic random environment}\label{applications}


{In this section we start with    a  stochastic  process  $(\s_t)_{t\geq 0}$, called \emph{dynamic random environment},  with  state space   $\O:= S^{\bbZ^d}$, $S$ being a compact Polish space. We assume it has    c\`adl\`ag paths in the   Skohorod space  $ D[\bbR_+; \O) $.
We will  then  introduce two random walks  $(X_t)_{t \geq 0}$ and $(X_t^{(\e)})_{t \geq 0}$,  on $\bbZ^d$,  whose jump rates depend on the dynamic  environment.
   The random walk  $(X_t^{(\e)})_{t \geq 0}$ will be thought of as a perturbation of  $(X_t)_{t \geq 0}$ and the parameter $\e$ will  quantify  the perturbation. More precisely, we give conditions in terms of Markov generators {ensuring} that the  process ``environment viewed from the walker $X^{(\e)}_t$'' (i.e. $\t_{X^{(\e)}_t} \s_t$) is a perturbation of the process  ``environment viewed from the walker $X_t$'' (i.e. $\t_{X_t} \s_t$).  In the above notation,     $\tau_x$ denotes   the translation  operator on $\O$ such that $\tau_x\eta(y)=\eta(x+y)$ for $x,y\in\bbZ^d, \eta\in\O$.}

{
 In Subsection \ref{torta1} we  introduce the main mathematical objects under  investigation and  our assumption. In Subsection \ref{torta3} we  present our main results concerning random walks in dynamic random environments, while  in Subsection \ref{torta2} we  discuss examples and collect some comments. 
}

\subsection{{Processes  and assumptions}}\label{torta1}
 \begin{assumptionA} \label{ass:base1}The dynamic random environment is a {Feller} process and is stationary w.r.t.\@ a { probability measure $\mu$ on $\O$.  Moreover, $\mu$ is translation invariant.}
 \end{assumptionA}
{We denote by  $\bigl(S_{\rm env}(t)\bigr)_{t \geq 0 }$ the Markov semigroup in $L^2(\mu)$ associated  {with} the dynamic random environment, and by   $L_{\rm env}: \cD( L_{\rm env} ) \subset L^2(\mu) \to L^2(\mu)$ 
  the corresponding generator.    In particular, given $f \in L^2(\mu)$, it holds $(S_{\rm env}(t) f) (\s) :=  \bbE^{\rm  env}_\s\bigl[ f(\s_t) \bigr] $ $\mu$--a.s., where  $\bbE^{\rm env}_\s$ is the expectation 
  for the dynamic random environment starting at $\s$.
  }
  

\begin{assumptionA}\label{ass:base2}The dynamic random environment commutes with translations, i.e.\
\begin{equation}\label{retour}
S_{\rm env}(t)(f\circ\tau_x)=(S_{\rm env}(t)f)\circ\tau_x, \qquad \forall f \in L^2(\mu)\,,\; t \geq 0\,.
\end{equation}
{Moreover, the generator  $L_{\rm env}$   satisfies the Poincar\'e inequality, i.e.\ there exists $\g>0$ such that 
 \begin{equation}\label{poincareXX}
\g \|f\|^2 \leq - \mu ( f L_{\rm env} f ) \qquad \forall f \in \cD (L) \text{ with } \mu(f)=0\,.
\end{equation}}
\end{assumptionA} 
{We point out that \eqref{poincareXX} is equivalent to the bound  
$ \| S _{\rm env} (t)  f - \mu(f) \|\leq e^{-\g t} \|f -\mu(f) \|$ for all $ t \geq 0$ and $ f \in L^2(\mu)$, $\| \cdot \|$ being the norm in $L^2(\mu)$ (see Lemma \ref{candelina} in Appendix \ref{app_misc}).}

\smallskip

We now want to introduce two  random walks on $\bbZ^d$,  whose jump rates depend on the dynamic random environment. To this aim,  we require the following:

\begin{assumptionA}\label{ass:rates}
There are given continuous functions   $r_\e(y,\cdot)$, $r(y,\cdot)$  and  $\hat{r}_{\e}(y,\cdot) $ on $\O$, parametrized by   $y\in\bbZ^d$.   These functions are zero {for} $y=0$,  $r_\e(y,\cdot)$ and $r(y,\cdot)$ are nonnegative and 
 $r_\e(y,\cdot)$ can be decomposed as 
\begin{equation}\label{rateDec}
r_\e(y,\cdot):=r(y,\cdot)+\hat{r}_{\e}(y,\cdot)\,.
\end{equation}
We also require that, for some $n\geq 1$,   the above functions have finite $n$-th moment:
\begin{equation}
\label{nthMoment}
\sum_{y\in\bbZ^d}|y|^n\sup_{\eta\in\O}    r (y, \eta)   <\infty\,, \qquad \qquad 
\sum_{y\in\bbZ^d}|y|^n\sup_{\eta\in\O}   | \hat{r}_\e (y, \eta) |  <\infty\,.
\end{equation}
\end{assumptionA}
\medskip

Let now $(X_t)_{t\geq 0}$  be the continuous time  random walk on $\bbZ^d$ jumping from 
site $x\in \bbZ^d$ to site $x+y\in\bbZ^d$ at rate $r(y,\tau_{x}\eta)$, given that the dynamic random environment is in state $\eta \in \O$.  
 Due to dependence on the environment, such a random walk is not Markovian itself, but 
the joint process  $(\s_t, X_t)_{t\geq 0}$ on state space $\O\times \bbZ^d$ is a Markov process with formal  generator\footnote{The notation $L_{\rm rwre}$  is thought to stress that we are referring to the joint process describing both the random walk and the random environment. 
}
\begin{equation}\begin{aligned}\label{uRWRE}
L_{\rm rwre} f(\eta,x) &:= L_{\rm env} f(.,x)(\eta)+ \sum_{y \in \bbZ^d}r(y,\tau_{x}\eta) \big[f(\eta,x+y)-f(\eta, x)\big],\;\; (\eta,x ) \in \O \times \bbZ^d\,.
\end{aligned}\end{equation} 
We do not insist here with a precise description of the generator, since it will not be used in the sequel.  On the other hand, below  we will discuss carefully the generator of the process ``environment viewed from the walker''.
{Due to \eqref{nthMoment},  no explosion takes place and therefore}
   the random walk $(X_t)_{t\geq 0}$ is well defined (a universal construction is given in Section  \ref{san_valentino}). In what follows we write $P_{\eta,x}$ for the law on the c\`adl\`ag space $D(\bbR_+; \O\times\bbZ^d)$ of this joint process starting at $(\eta,x)$. 

\medskip

As in the construction of the joint Markov process in \eqref{uRWRE}, we define a new joint Markov process
 $(\s_t, X^{(\e)}_t)_{t\geq 0}$  on state space $\O\times \bbZ^d$ with formal generator:
  
\begin{equation}\begin{aligned}\label{RWRE}
L_{\rm rwre}^{(\e)} f(\eta,x) &:= L_{\rm env} f(.,x)(\eta)+ \sum_{y \in \bbZ^d}r_\e(y,\tau_{x}\eta) \big[f(\eta,x+y)-f(\eta, x)\big],\;\; (\eta,x ) \in \O \times \bbZ^d\,.
\end{aligned}\end{equation} 
In what follows we write $P^{(\e)}_{\eta,x}$ for the law on the c\`adl\`ag space $D(\bbR_+; \O\times\bbZ^d)$ of this joint process starting at $(\eta,x)$. 
We refer to this new walker $(X^{(\e)}_t)_{t\geq 0}$ as the \emph{perturbed walker}.

One of the most common {approaches} to study random motion in random media is to analyze the so called {\em environment seen by the walker}. In our case, we are interested {in} the Markov processes on $\O$ given by $\tau_{X_t}\s_t$ and $\tau_{X^{(\e)}_t}\s_t$, where $(\s_t, X_t)_{t \geq 0}$ and $(\s_t, X^{(\e)}_t)_{t \geq 0}$ are the joint Markov processes defined above.

{We write  $C(\O)$ for the space of real continuous functions on $\O$ endowed with the uniform norm. Since, by assumption,  the dynamic random environment is a Feller process, it has a well defined Markov semigroup on $C(\O)$, and we  denote by\footnote{{We denote consistently with curved $\cL$ generators on $C(\O)$ and with straight $L$ their version living in $L^2(\mu)$.}} $ \cL_{\rm env} : \cD( \cL_{\rm env} ) \subset C(\O ) \to C(\O)$ the  associated Markov generator}. \ We define  $\cL_{\rm jump}f (\eta) =  \sum_{y \in \bbZ^d}r(y,\eta)\big[f(\tau_{y}\eta)-f(\eta)\big]$ for $f \in C(\O)$ and $\hat{\cL}_\e f (\eta) =  \sum_{y \in \bbZ^d}\hat r_\e(y,\eta)\big[f(\tau_{y}\eta)-f(\eta)\big]$ for $f \in C(\O)$.
 Then, by Assumption \ref{ass:rates}, the operators  $\cL_{\rm jump},\hat{\cL}_\e : C(\O) \to C(\O)$  are well posed and bounded.
  
 \begin{assumptionA}\label{ass:feller}
 The environment seen from the unperturbed walker  $\bigl(\tau_{X_t}\s_t\bigr)_{t \geq 0}$ and the one seen from the perturbed walker $\bigl( \tau_{X^{(\e)}_t}\s_t\bigr)_{t \geq 0}$ are Feller  processes on $\O$ with generators on $C(\O)$  given respectively  by 
 $ \cL_{\rm env}+ \cL_{\rm jump}$ and  $\cL_{\rm env}+ \cL_{\rm jump}+\hat{\cL}^{(\e)}$, both  having domain $ \cD ( \cL_{\rm env})$.
  \end{assumptionA}
  The above assumption is typically satisfied in all common applications:
  \begin{Proposition}\label{ale_giacomo}
 Suppose  that $\cL_{\rm env}$ is the closure of a Markov pregenerator $\bbL$  as in \cite[Def.2.1, Chp.I]{L}, satisfying the
criterion in \cite[Prop.2.2, Chp.I]{L}.  Then $ \bbL+ \cL_{\rm jump}$ and $\bbL+ \cL_{\rm jump}+ \hat{\cL}^{(\e)}$ are Markov pregenerators,  whose closures  are Markov generators  of  Feller processes (cf.\@ \cite[Def.2.7, Chp.I]{L}). The resulting Markov  generators  are given respectively  by the operators $ \cL_{\rm env}+ \cL_{\rm jump}$ and  $\cL_{\rm env}+ \cL_{\rm jump}+\hat{\cL}^{(\e)}$, both having domain $ \cD ( \cL_{\rm env})$.
  \end{Proposition}

The proof of the above proposition is similar to the proof of  \cite[Lemma 2.1]{DF}. The interested reader can find the proof of  Prop. \ref{ale_giacomo} in \cite[Appendix A]{ABFv2}.

\begin{assumptionA}\label{ass:stat}
The environment seen by the  unperturbed walker $\bigl(\tau_{X_t}\s_t\bigr)_{t \geq 0}$ has invariant distribution $\mu$.
\end{assumptionA}
\begin{remark}
Due to Assumption \ref{ass:base1}, Assumption \ref{ass:stat} is equivalent to the fact that   $\mu( \cL_{\rm jump} f)=0$ for any $f \in C(\O)$ (or for any $f$  in a dense subset of $C(\O)$, since  {$ \cL_{\rm jump}$ is a bounded operator due to \eqref{nthMoment}}).
\end{remark}

We can state our last main assumption, which is indeed related to the perturbative approach. Consider the operator $\hat L_\e: L^2(\mu) \to L^2(\mu)$ defined  as
\begin{equation} \label{RateTurbati}
 \hat L_\e f(\eta) :=\sum_{y \in \bbZ^d}\hat{r}_{\e}(y,\eta)\left[f(\tau_y\eta)-f(\eta)\right]\,,\qquad f \in L^2(\mu)\,.
\end{equation}
It is indeed  a  bounded operator in $L^2(\mu)$.  For example,  by Schwarz inequality and by  Assumption \ref{ass:rates},  given $f \in L^2(\mu)$
 we can write $$ \mu\left( \bigl[ \hat L_\e f \bigr]^2\right) \leq   \Big[\sum _{y\in \bbZ^d} \sup _\eta |\hat r_ \e (y, \eta) | \Big] \sum _{y \in \bbZ^d } \sup_{\eta} |\hat r_\e (y, \eta) | \mu ( [ f( \t_y \cdot )-f]^2)\,,
$$ and by the translation invariance of $\mu$  we conclude that
\begin{equation}\label{epifania}
\| \hat L_\e \|\leq {2}  \sum _y \sup _\eta |\hat r_ \e (y, \eta) | \,.
\end{equation}

\begin{assumptionA}\label{asso3}
The operator $\hat L_\e$ has norm $\e:= \|\hat L_\e\|$  satisfying $ \e < \g$, where $\g$ has been introduced in Assumption \ref{ass:base2}.
\end{assumptionA}


\subsection{{Some examples}}\label{torta2}
{\emph{Dynamic environments.}
Natural examples of environments satisfying our assumptions are given by Interacting Particle Systems (IPSs) with state space $\Omega=\{0,1\}^{\bbZ^d}$.
A first class of such IPSs is that of translation invariant \emph{stochastic Ising models} in a ``high--{temperature}'' regime (see \cite[Thm.4.1]{Ma} and \cite[Thm.4.1, Chp.I]{L}), among which, the simplest case is the independent spin-flip dynamics. The latter is the Markov process with generator $\mathcal{ L}_{\rm env}f(\sigma)=\gamma\sum_{x\in\bbZ}f(\sigma^x)-f(\sigma)$, where $\gamma>0$, and $\sigma^x$ is the configuration obtained by $\sigma\in\Omega$ by flipping the spin at $x$.
As a variant of these processes, one could consider some Kawasaki dynamics superposed to a high--noise spin--flip dynamics.} {When the exponential convergence of the Markov semigroup holds in the stronger $L^\infty$--norm one could also apply \cite[Sec. 3]{AdHR} to derive some of the results presented here (as the existence of the limiting velocity). On the other hand, 
  several of our results have not been derived in the existing literature, even under the assumption of $L^\infty$--convergence; moreover, there are several models where the Poincar\'e inequality holds while the log-Sobolev inequality is violated  or has not been proved.}
{One of the motivations} which prompted the present study {is to consider the class of so--called} Kinetically Constrained Spin Models (KCSMs), for which \eqref{poincareXX} was proved in great generality (in the ergodic regime) in \cite{CMRT}. Their generator is given by $\mathcal{ L}_{\rm env}f(\sigma)=\sum_{x\in\bbZ}c_x(\sigma)(\rho(1-\sigma(x))+(1-\rho)\sigma(x))\left[f(\sigma^x)-f(\sigma)\right]$ with $\rho\in (0,1)$ and $c_x$ encodes a kinetic constraint which should be of the type ``there are enough empty sites in a neighbourhood of $x$''. We refer to \cite{CMRT} for precise conditions that the constraints need to satisfy and identification of the regime where \eqref{poincareXX} is satisfied. Examples of constraints include the FA-$j$f model, where $c_x(\sigma)=\mathbf{1}_{\sum_{y\sim x}(1-\sigma(y))\geq j}$ with $j\leq d$, or generalized East processes $c_x(\sigma)=1-\prod_{i=1}^d\sigma(x+e_i)$ with $(e_i)_{i=1,\cdots,d}$ the canonical basis of $\bbR^d$. The presence of the constraint gives rise to a number of difficulties {as for instance} the lack of attractivity. Consequently, most of the general existing results, as e.g. \cite{AdHR,BZ,BMP,BK,DKL,RAS2,RV}, do not apply to this class. 

\noindent{\emph{Random walks.}
We give here three simple though non--trivial examples of different nature for which our results apply. 
The simplest case is when the unperturbed walker is not present: that is, 
$r(y , \cdot )\equiv 0$ for all $y \in \bbZ^d$. Then environment and environment from the unperturbed walker coincide and all our results 
are valid for any random walk choice satisfying our basic assumptions, provided that the rates are small enough. 
As a second case, we can consider random walks obtained as perturbations of simple symmetric random walks, that is, $r(y,\cdot)=1/2$ for $y=\pm 1$ and $0$ else, again, provided that Assumption \ref{asso3} is in force. An interesting example is for $r_\e(y,\eta)=\pm\e(2\eta(0)-1)\mathds{1}_{\{y=\pm 1\}}$ for which the resulting random walk has the tendency to stick to the space-time interfaces between empty and occupied regions in the environment. A more detailed analysis of this walk on the East model, mainly based on the results in this work, can be found in \cite{ABF}. 
The last case is when the unperturbed walk depends effectively on the underlying environment, for which, in order to check the crucial Assumption \ref{ass:stat}, the specific choice of the environment is essential. For example, if the latter is given by a KCSM, as }{in \cite{JKGC}, one could consider a probe particle driven by a constant external field in the KCSM started from a stationary distribution $\mu$ left invariant by the non--driven prove. In the one-dimensional case one possibility is $r(\pm 1,\eta)=(1-\eta(0))(1-\eta(\pm 1))$, $r_\e(\pm 1,\eta)=(1-\eta(0))(1-\eta(\pm 1))\tilde{r}_\e(\pm 1)$, $\tilde{r}_\e( 1)=2/(1+e^{-\e})=e^\e\tilde{r}_\e( -1)$, the other rates are zero and $\e$ is small enough.} 

\subsection{Main results}\label{torta3}
In the rest of this section, we suppose Assumptions {\ref{ass:base1},...,\ref{asso3}}  to be satisfied without further mention. 

{Concerning  the environment seen by the walker $( \t_{X_t} \s_t)_{t \geq 0}$, we denote by  $\bbP_\nu$ its law on $D( \bbR_+; \O)$, and by $\bbE_\nu$ the associated expectation,  when  the initial distribution is $\nu$ (if $\nu=\d_{\eta}$, we  simply write $\bbP_\eta$ and $\bbE_\eta$). We denote by $S(t)$ its  Markov  semigroup on $L^2(\mu)$, i.e.\ 
$( S(t) f)(\eta) := \bbE_\eta \bigl( f( \eta_t)\bigr) $ $\mu$--a.s., and we write $L_{\rm ew}$ for its infinitesimal generator. For the perturbed version  $( \t_{X^{(\e)} _t} \s_t)_{t \geq 0}$   we use analogously the notation $\bbP_\eta^{(\e)}$, $\bbE_\eta^{(\e)}$
  for the law and the expectation. Moreover, we define 
   $(S_\e(t))_{t \geq 0}$ as the semigroup in $L^2(\mu)$ with  infinitesimal generator 
   $L_{\rm ew}^{(\e)}= L_{\rm ew}+ \hat{L}_\e$ (see Section \ref{auto} for a detailed discussion). As proved in 
   Section \ref{auto}, $(S_\e(t) f) (\eta) = \bbE^{(\e)}_\eta ( f(\eta_t) )$ $\mu$--a.s. at least for bounded continuous functions $f$. 
   }

{Given $t \geq 0$ we define iteratively the operators $S_\e ^{(n)}(t)$ as 
$ S_\e^{(0) }(t):= S(t)$,  $S_\e^{(n+1)} (t):=\int_0^t  S(t-s) \hat L_\e S_\e^{(n)}(s) ds$. These operators enter in the Dyson expansion $S_\e(t)= \sum_{n=0}^\infty S_\e ^{(n)}$ discussed in detail in Section \ref{viva_segovia}. }


\begin{TheoremA}[Asymptotic perturbed stationary state and velocity]\label{late}
  $ \,\,$
\begin{itemize}
\item[(i)] The {environment seen by the perturbed walker   admits a unique distribution  $\mu_\e $ on $\O$  which is invariant and absolutely continuous w.r.t.\@ $\mu$}. {Whenever the environment seen by the perturbed walker  has initial distribution  absolutely continuous w.r.t.\@ $\mu$, its distribution at time $t$ weakly converges to $\mu_\e$ as $ t \to \infty$}. Moreover, $\mu_\e$  is  ergodic w.r.t.\@ time--translations  and 

  \begin{equation}\label{muinfinity_bisXXX}
\mu_\e(f) = \mu(f)+\sum _{n=0}^\infty \int_0^\infty \mu\left( \hat L_\e S_\e^{(n)}(s) f \right) ds\,, \qquad f \in L^2(\mu)\,, 
\end{equation}
where
$\int_0^\infty \bigl|  \mu\left( \hat L_\e S_\e^{(n)}(s)f \right) \bigr| ds
\leq    (\e  /\g)^{n+1} \| f-\mu(f) \|$.

  \item[(ii)]    If the additional condition 
  \begin{equation}\label{ellipticity}
  r(y,\eta)>0 \; \; \Longrightarrow\;\; r_\e(y,\eta)>0  
  \end{equation} 
  is satisfied, 
  then 
 $\mu$ and $\mu_\e$ are mutually absolutely continuous.
 
 {Alternatively, if there exist subsets $V, V_\e \subset \bbZ^d$ such that 
 \begin{itemize}
 \item[(a)] $r(y, \eta )>0 $ iff $y \in V$, 
 \item[(b)] $\hat r_\e(y, \eta )>0 $ iff $y \in V_\e$, 
 \item[(c)] each vector in $V$ can be written as sum of vectors in $V_\e$,
 \end{itemize}
  then $\mu$ and $\mu_\e$ are mutually absolutely continuous.}
 \item[(iii)] {If \eqref{nthMoment} holds with $n=2$}, then  defining  $v(\e):= \mu_\e (j^{(\e)})$ with  $j^{(\e)}(\eta):=\sum_{y\in\mathbb{Z}^d}y r_\e(y,\eta),\eta\in\O$,   it holds
\begin{equation}\label{veloce}
P^{(\e)}_{\eta,0}\Big(\lim _{t \to \infty} \frac{X^{(\e)}_t}{t}= v(\e)\Big)=1
\end{equation} 
for 
$\mu_\e$--a.e.\@ $\eta$ and for $\eta$ varying in a set of $\mu$--probability larger than $1-\e^2/(\g-\e)^2$. If {$\mu$ and $\mu_\e$ are mutually absolutely continuous as in   Item (ii)}, then \eqref{veloce} holds for $\mu$--a.e.\@ $\eta$.
\item[(iv)] The asymptotic velocity $v(\e)$ can be expressed by a series expansion in $\e$ as 
\begin{equation}\label{speed}
v(\e)=\mu(j^{(\e)})+\sum_{n=0}^\infty\int_0^\infty \mu( \hat L_\e S_\e^{(n)}(s) j^{(\e)} )ds.\end{equation}
Moreover,      $  \bigl| \mu(  \hat L_\e S_\e^{(n)}(s) j^{(\e)})\bigl |  \leq    \e^{n+1} e^{-\g s } s^n \|j^{(\e)}\|_\infty /n!$  for all  $n\geq 0$.  
\end{itemize}
\end{TheoremA}
{\begin{remark} Further properties on the distribution $\mu_\e$ and on the semigroup $S_\e(t)$  are stated, in a more general context,  in Section \ref{viva_segovia}  (see in particular Proposition \ref{spezzatino} and formulas   \eqref{convergencetomean}, \eqref{barbiere}, \eqref{benedicte}, \eqref{expansiong}  and  \eqref{arachidi_bis}   in Theorem \ref{teo_invariante}).
\end{remark}} 

The proof of Theorem \ref{late} is given in Section \ref{RWproof_mela}.


 \begin{TheoremA}\label{porte-bonheur} 
Suppose that $\mu$   has the following decorrelation property: given functions  $f,g$ with bounded support, we have
\begin{equation}\label{zietto} \lim _{|x|\to \infty}  {\rm Cov}_\mu(f,\t_x g)=0\,.\end{equation}
Then, for any   local function $f$, it holds
\begin{equation}\label{belgio}
\lim _{|x|\to \infty}  \mu_\e(\t_x f)= \mu(f) \,.
\end{equation}
\end{TheoremA}
The proof of Theorem \ref{porte-bonheur} is given in Section \ref{tao}.

Under stronger conditions, we can estimate the decay of $|\mu_\e(\t_x f)- \mu(f) |$.  To this aim we fix some notation and terminology.
Given $x \in \bbZ^d$ and $\ell >0$, we introduce the uniform box  $B(x, \ell)= \{ y \in \bbZ^d \,:\, | x-y|_\infty \leq \ell\} $. If $x=0$, we simply write $B(\ell)$.

 \begin{definition}
The  stationary   process \emph{dynamic random environment} with generator $L_{\rm env}$ and 
   initial distribution $\mu$  has \emph{finite speed of propagation}  if  there exists a function $\a: \bbR_+ \to \bbR_+$ vanishing at infinity (i.e.\@ $\lim_{u \to \infty} \a (u)=0$)  and a constant $C>0$ such that 
\begin{equation}
\left|\bbE^{\rm env} _\mu[XY]-\bbE_\mu^{\rm env} [X]\bbE_\mu^{\rm env} [Y]\right|\leq  \a(d(\L,\L'))
\end{equation}
for any pair of random variables $X,Y$ bounded  in modulus by one and for any pair of sets $\L, \L' \subset \bbZ^d$, such that (for some $t \geq 0$)   $X$ is determined by  $\bigl( \eta_s(x)\,:\,  0 \leq s \leq t, \; x \in \L\bigr) $, $Y$  is determined by  $\bigl( \eta_s(x)\,:\,  0 \leq s \leq t, \; x \in \L'\bigr) $, and 
$d(\L, \L')= \min \{ |x-x'|_\infty\,:\, x \in \L, \, x' \in \L'\} \geq Ct$.
 \end{definition}

The above property is satisfied for example by  many interacting particle systems {on $\bbZ^d$, in particular it is fulfilled  if the transition rates are bounded and  have finite range, as can be easily checked  from the graphical construction (see e.g.  {\cite[Chap. III, Sec. 6]{L}, \cite[Sec. 3.3]{Ma}).}   }

\begin{TheoremA}[Quantitative approximation of $\mu_\e$ by $\mu$ at infinity]\label{canicule}  {In addition to our main assumptions, }  assume the following properties:
\begin{enumerate}
\item[(i)]  translation invariance of the unperturbed dynamics, i.e.\@ $S(t)\bigl( f\circ \t_x ) =\bigl(S(t) f\bigr) \circ\tau_x$,  for any local function $f$,  $x \in \bbZ^d$ and $ t \geq 0$,
\item[(ii)]  the stationary   process  with generator $L_{\rm env}$ has finite speed of propagation with  $\alpha(u)\leq e^{-\theta u }$  for some $\theta>0$,
\item[(iii)]  $r,\hat r_\e $ have finite range, i.e.\@ $\exists R>0$ such that  $r (z,\cdot) \equiv 0 $ and $\hat r_\e(z, \cdot)\equiv 0$ if $z \not \in B(R)$ and such the support of $r (z, \cdot)$ and  $\hat r_\e(z, \cdot) $  is included in $B(R)$.
\end{enumerate}
Then there exists $\theta'>0$ (depending on $\e$ and $\g$) such that, for any local function  $f: \O \to \bbR $, it holds
\begin{equation}\label{tropchaudmu}
|\mu_\e(\t_x f )-\mu(f)|\leq C(f,\e,\g)e^{-\theta'|x|_\infty}\,,
\end{equation}
where $C(f, \e, \g)$ is a finite constant depending only on $f,\e,\g$.
\end{TheoremA}

\begin{remark}
One could prove Theorem~\ref{canicule} without Assumption (i), and  also its analogue for different decays in the finite speed propagation property, but the treatment would become very technical. Hence we  have preferred to restrict  to the above  simpler case.
\end{remark}

The next lemma  gives a sufficient condition for  Assumption (i) in Theorem \ref{canicule}:
\begin{Lemma}\label{commutare}
Assume~\eqref{retour} and that 
the unperturbed random walk is decoupled from the environment, \i.e.\@ $r(y,\eta)$ does not depend  on $\eta$ for any $y \in \bbZ^d$. Then the assumption  in Item (i) of Theorem \ref{canicule} is satisfied.
\end{Lemma}

The proofs of Theorem \ref{canicule} and Lemma \ref{commutare} are given in Section \ref{tao}.

\smallskip 

\medskip

Our next result focuses on gaussian fluctuations of the random walk:
 \begin{TheoremA}[Invariance principle for the perturbed walker]\label{QCLT} 
 
 \noindent
 $(i)$  Suppose that  {\eqref{nthMoment} holds with $n=2$.}
Then there exists a symmetric non--negative $d \times d$ matrix $D_\e$ such that,  under {$\int \mu_\e(d \eta) P^{(\e)}_{\eta,0}$}, as $n\to \infty$ the rescaled process
\begin{equation}
\frac{X^{(\e)}_{nt}-v(\e)nt}{\sqrt{n}} 
\end{equation}
converges weakly to a  Brownian motion with covariance matrix $D_\e$.

\noindent
$(ii)$
Suppose in   addition that  $L_{\rm env}$ and $L_{\rm ew}$ are self--adjoint in $L_2(\mu)$, equivalently that $L_{\rm env} $ is self--adjoint and   $r$ satisfies 
\begin{equation}\label{DetailedBal} r(y,\eta)=r(-y,\tau_y\eta).\end{equation}
{Moreover, assume that  \eqref{nthMoment} holds with $n=4$. }Then
 the limiting Brownian motion has non-degenerate covariance matrix {for  $\beta(\e)$ small enough, where}\footnote{Note that by \eqref{epifania}, a small $\beta(\e)$ implies that $\e$ is small. 
}
\begin{equation}\label{betacarotene} \beta(\e):=   \sum _{y\in \bbZ^d}   |y| \sup _\eta  | \hat r_ \e (y, \eta) | \,.\end{equation}
\end{TheoremA}
The proof of Theorem \ref{QCLT} is given in Sections \ref{RWproof_pera} and \ref{NonDeg}. 

\section{$L^2$--perturbation of stationary Markov processes}\label{viva_segovia}

{As already mentioned, the  derivation of the results presented in Section \ref{applications} is based - between others - on a perturbative approach. In this section,  starting from the 
Dyson--Phillips expansion of the Markov semigroup,   we derive some results on perturbations of stationary  Markov processes satisfying the Poincar\'e inequality. We will focus on the  
perturbed invariant distribution, the perturbed Markov  semigroup, the LLN and invariance principle for additive functionals of the perturbed process. We have stated these results in full generality, while at the beginning of  Section \ref{RWproof_mela} we explain how the random walks in dynamic random environments analyzed in Section \ref{applications} fit {into} this general scheme.}

We fix a metric space $\O$, which is thought of as a measurable space endowed with the $\s$--algebra of its Borel sets.  We consider a Markov process  with state space $\O$  and 
with c\`adl\`ag paths  in the {Skorokhod} space  $D( \bbR_+; \O)$.  
We write $( \eta_t)_{t \in \bbR_+}$ for a generic path,  denote by $\bbP_\nu$ the law  on $D( \bbR_+;\O)$ of the process with initial distribution $\nu$, and  by $\bbE_\nu$ the associated expectation. If $\nu= \d_\eta$, $\eta \in \O$, we simply write $\bbP_\eta$, $\bbE_\eta$.
We suppose the process to have an \emph{invariant distribution} $\mu$ on $\O$. Then the  family of  operators $ S(t)f (\eta):= \bbE_\eta \bigl[ f(\eta_t) \bigr]$, $t \in \bbR_+$, gives 
a contraction semigroup in $L^2(\mu)$, which is indeed strongly  continuous\footnote{Strongly continuous semigroup are often called \emph{$C_0$--semigroups}} 
in $L^2(\mu)$ (see Lemma \ref{SC} in Appendix).  We write $L$ for its infinitesimal generator  (in $L^2(\mu)$) and $\cD(L)$ for the corresponding domain.  In what follows we denote by $\|\cdot \|$ the norm in $L^2(\mu)$ and by $\mu(f)$ the $\mu$--expectation of an arbitrary function $f$.
We assume that $L$ satisfies the \emph{ Poincar\'e inequality}, i.e.\@ for some $\g>0$
 \begin{equation}\label{poincare}
\g \|f\|^2 \leq - \mu ( f Lf ) \qquad \forall f \in \cD (L) \text{ with } \mu(f)=0\,.
\end{equation}
Note that  the above Poincar\'e inequality  is equivalent to the bound (cf.\@ Lemma \ref{equitalia} in Appendix)
  \begin{equation}\label{banana}
 \| S (t) f - \mu(f) \|\leq e^{-\g t} \|f -\mu(f) \| \, \qquad \forall t \geq 0 \,, \; f \in L^2(\mu)\,.
\end{equation}
If $\mu$ is reversible w.r.t.\@ $L$, then \eqref{poincare} corresponds to requiring that $L$ has  spectral gap bounded by $\g$ from below.

\medskip

 Next, for a given fixed parameter  $\e>0$, we consider a new Markov process on  $\O$ and call  $\bbP_\nu^{(\e)}$ its law on $D( \bbR_+; \O)$ when starting with distribution $\nu$, and $\bbE_\nu^{(\e)}$ the associated expectation. 
In the sequel we refer to this new Markov process as the \emph{perturbed process}.
We introduce a bounded operator $\hat{L}_\e : L^2(\mu) \to L^2(\mu)$, with $\e:= \|\hat{L}_\e\|$, and set
 \begin{equation}\label{turbato}
 L_\e:=L+\hat{L}_\e\,,\qquad \cD(L_\e):= \cD(L)\,.
 \end{equation} 
 It is known (cf.\@ \cite[Thm. 1.3, Chp. III]{EN}) that the operator $L_\e = L + \hat L_\e$ with domain $D(L_\e)= D(L)$ is the generator of a strongly continuous semigroup $(S_\e (t) ) _{t \geq 0}$ on  $L^2 (\mu)$. Moreover, it holds $ S_\e (t) = e^{t L_\e}$, where the exponential of the operator $L_\e$ is defined in \cite[Ch. IX, Sec. 4]{K} (cf.\@ Problem 49 in \cite{RS2}[Ch. X]).

We fix our basic assumptions:
\begin{assumptionA}\label{ferragosto} The unperturbed Markov process has invariant and ergodic distribution $\mu$. The  generator $L$  of the    $L^2(\mu)$--semigroup $S(t)$, $t \in \bbR_+$,  satisfies  the Poincar\'e inequality \eqref{poincare}. Moreover,  considering the semigroup $S_\e (\cdot)$ with generator $L_\e=L+ \hat L_\e$ and the perturbed Markov process, 
it holds \begin{equation}\label{ghiacciolo}
S_\e (t) f (\eta)= \bbE^{(\e)}_\eta \bigl( f(\eta_t) \bigr) \,, \qquad \mu{\rm -a.s.}\,, \qquad \forall f \in C_b(\O)\,,
\end{equation}
 where 
we denote by $C_b(\O)$ the space of bounded continuous real functions on $\O$.
\end{assumptionA}

\begin{remark}\label{compleanno} The above ergodicity of $\mu$ has to be thought w.r.t.\@ time translations, i.e.\@ any Borel set $A \subset D(\bbR_+, \O)$ which is left invariant by any time translation\footnote{Time translation $\theta_t : D(\bbR_+, \O) \to D( \bbR_+, \O)$ is defined as $(\theta_t \eta)_s := \eta_{t+s}$.} $\theta_t$ has $\bbP_{\mu}$--probability equal to $0$ or $1$. Due to Theorem 6.9 in \cite{V} (cf.\@ also \cite[Chapter  IV]{Ros}), this is equivalent to the following fact: $\mu (B)\in \{0,1\}$ if $B$ is a Borel subset of $\O$ such that 
     $ \mathds{1}_B(\eta_0)= \mathds{1}_B(\eta_t) $  $\bbP _{\mu}$--a.s.\@ for any $t \geq 0$.  
     Note that for such a subset $B$ it holds $S(t) \mathds{1}_B = \mathds{1}_B$ $\mu$--a.s.. This observation allows to deduce the ergodicity of $\mu$ from the bound \eqref{banana}, since we assume that $S(\cdot)$ satisfies  the Poincar\'e inequality. Hence, the explicit hypothesis of $\mu$ ergodic could be removed from Assumption \ref{ferragosto}.
\end{remark}
 
In the following lemma we discuss a case, useful in applications, where the above property \eqref{ghiacciolo} is fulfilled (the proof is postponed to Section \ref{golden_gate}). The lemma  covers numerous applications, e.g. interacting particle systems (cf. \cite{L}, in particular Chp. IV.4 there):
\begin{Lemma}\label{ponte} Suppose that $\O$ is compact and that the perturbed Markov process is Feller on $C(\O)$ endowed with the uniform norm. Consider the induced Markov semigroup $\tilde S_\e(t) $, $ t \in \bbR_+$, on $C(\O)$: $ \tilde S_\e(t) f (\eta):=\bbE^{(\e)}_\eta \bigl( f(\eta_t) \bigr) $ for $f \in C(\O)$. Call  $\tilde L_\e: \cD( \tilde L_\e ) \subset C(\O) \to C(\O)$  its infinitesimal generator. Suppose that $\tilde{L}_\e$
has  a core $\cC_\e \subset \cD( \tilde L_\e) \cap \cD(L_\e)$ such that    $\tilde L_\e f =  L_\e f$ for all 
$f\in \cC_\e  $. Then  identity \eqref{ghiacciolo}  is satisfied.
\end{Lemma}

We recall, cf.\@ \cite[Cor. 1.7 and Eq. (IE$^{*}$), Chp. III]{EN},   the  so called \emph{variation of parameters formula}: for any $ f \in L^2(\mu)$ it holds 
\begin{equation}\label{dyson}
\begin{split}
S_\e(t)f  & =S (t) f + \int_0^t S(t-s) \hat L_\e S_\e(s)f  ds \\
&  =S (t) f + \int_0^t  S_\e  (s) \hat L_\e S(t-s)f  ds\,,
\end{split}
\end{equation}
where the above integrals have to be understood in $L^2(\mu)$.


Given $t \geq 0$ we define iteratively the operators $S_\e ^{(n)}(t)$ as 
\begin{equation}\label{Sn}
S_\e^{(0) }(t):= S(t), \quad S_\e^{(n+1)} (t):=\int_0^t  S(t-s) \hat L_\e S_\e^{(n)}(s) ds\, =\int_0^t S_\e^{(n)}(s)\hat L_\e  S(t-s)  ds \,.
\end{equation}
The equivalence of the two forms of $S_\e^{(n+1)}$ in \eqref{Sn} can be checked by induction {(see  \cite[App. A]{ABFv2})}. As explained in \cite[Chp. III]{EN}, $S_\e ^{(n)} (\cdot)$ is a continuous function from $\bbR_+$ to the space $ \cL ( L^2(\mu))$ of bounded operators in $L^2(\mu)$. Moreover, the \emph{Dyson--Phillips expansion} holds: 
\begin{equation}\label{DP}
S_\e (t)= \sum_{n=0}^\infty S_\e ^{(n)}(t) \,,  \qquad t \geq 0 \,,
\end{equation}
where the series converges in the operator norm  of $ \cL ( L^2(\mu))$, even uniformly as $t$ varies in a {bounded} interval.


 By means of the Poincar\'e inequality, we can derive more information on the Dyson--Phillips expansion and on the semigroup $(S_\e(t) )_{t \geq 0}$:
 \begin{Proposition}[Dyson--Phillips expansion]\label{spezzatino} Let $\e <\g$, for any $f \in L^2(\mu)$ and $t\geq 0$ it holds
\begin{equation}\label{volare}
 \| S_\e(t) f - \sum _{n=0}^{k-1} S_\e^{(n)} (t) f \|\leq  (\e/\gamma )^{k} \left(\frac{2\g}{\g- \e }\right)  \| f - \mu(f) \|  \,, \qquad \forall k \geq 1.
 \end{equation}
 \end{Proposition}
 The above proposition is proven in Section \ref{primino}.

\begin{TheoremA}[Invariant measure]\label{teo_invariante}
Let   Assumption \ref{ferragosto} be satisfied and let $\e <\g$. Then  there exists a probability measure $\mu_\e$ on $\O$ with the following properties: 
 \begin{enumerate}[(i)]

 
 \item\label{a} Consider the  perturbed Markov process with initial distribution $\nu$   absolutely continuous w.r.t.\@ $\mu$. 
 Then its  distribution at time $t$ weakly converges to   $\mu_\e$ as $t \to \infty$. 
  \item\label{iii}  For each $f \in L^2(\mu)$ {it holds}   
  \begin{equation}\label{muinfinity_bis}
\mu_\e(f) = \mu(f)+\sum _{n=0}^\infty \int_0^\infty \mu\left( \hat L_\e S_\e^{(n)}(s) f \right) ds\,,
\end{equation}
{where
\begin{equation}\label{ikea} 
\int_0^\infty \bigl|  \mu\left( \hat L_\e S_\e^{(n)}(s)f \right) \bigr| ds
\leq    (\e  /\g)^{n+1} \| f-\mu(f) \|\,.
\end{equation} 
} 
Moreover, for $t\geq 0$, the following estimates hold:
\begin{align}
& \| S_\e(t) f - \mu( S_\e(t)f ) \|\leq e^{-(\g-\e )t} \| f-\mu(f)\| \,,  \label{convergencetomean}\\
& \bigl|\mu( S_\e(t)f )  - \mu_\e (f) \bigr| \leq \frac{ \e}{\g-\e}   e^{-(\g-\e )t} \| f-\mu(f)\|\,,\label{barbiere}\\
& 
\left|\mu_\e(f)-\mu(f)\right|\leq
\frac{\e}{\g-\e} \|f-\mu(f)\|\,.\label{benedicte}
\end{align}

 \item\label{ii} $\mu_\e$ is the unique { distribution  which  is both absolutely continuous  w.r.t.\@ $\mu$ and  invariant  for the perturbed Markov process}. The    Radon--Nykodim derivative $h_\e :=d\mu_\e /d\mu$ belongs to $L^2(\mu)$ 
 and admits the expansion\footnote{We denote by $A^*$ the adjoint of the operator $A$ on $L^2(\mu)$} 
 \begin{equation}\label{expansiong}
h_\e=\mathds{1}+\sum_{n=1}^\infty\int_0^\infty H^{(n)}_\e(t)\mathds{1} dt,
\end{equation}
where $H^{(n)}_\e(t) := [S_\e^{(n-1)}(t)]^*  \hat L_\e^*,\, n\geq 1,$ are bounded operators on $L^2(\mu)$ satisfying the recursion:
\begin{equation}\label{defH}
H^{(n+1)}_\e(t)=\int_0^t ds\,H^{(n)}_\e(s)S^*(t-s)\hat L_\e^*=\,\int_0^t ds\,S^*(t-s)\hat L_\e^*H^{(n)}_\e(s).
\end{equation}
Moreover,  it holds  $\| h_\e-\mathds{1}\| \leq \frac{\e}{\g-\e}$.

\item\label{v}  Suppose that for  any  $t>0$  and for any  measurable   $B \subset \O$  it holds 
\begin{equation}\label{ho_fame}
\mu \bigl( \{  \eta \in B^c \,:\, \bbP_\eta^{(\e) } ( \eta_t \in B)= 0  \text{ and }\bbP_\eta ( \eta_t \in B)>0 \}\bigr)=0\,.
\end{equation}
 Then  also $\mu$ is absolutely continuous w.r.t.\@ $\mu_\e$.

 \item\label{iv} For any  $f\in L^\infty (\mu)$ it holds
  \begin{equation}\label{arachidi_bis}
\| S_\e(t) f - \mu_\e(f) \|_{\e} \leq   \bigl( \frac{\g}{\g-\e} \bigr) ^{3/2}e^{ -\frac{ \g- \e }{2}  t }   \| f-\mu(f) \|_\infty \,, \qquad t \geq 0\,,
\end{equation}
where $\| \cdot\|_\e $, $\| \cdot \|_\infty$  denote the norm in $L^2(\mu_\e)$ and $L^\infty(\mu)$ respectively.

\item \label{i} $\mu_\e $ is ergodic w.r.t.\@ time--translations, as in Remark \ref{compleanno}.
 \end{enumerate}

\end{TheoremA}
The proof of the above theorem is given in Section \ref{assoluto}

\begin{remark}\label{KOmark}
Theorem \ref{teo_invariante} presents some intersection with \cite[Thm. 2.2 and Thm. 4.1]{KO}. There the authors consider also unbounded perturbations satisfying some sector condition and the analysis is not based on the Dyson--Phillips expansion. In particular, in \cite{KO}  the content of Theorem \ref{teo_invariante}--(i) is obtained only for $\nu \ll \mu$ with $d\nu / d \mu \in L^2(\mu)$ (while here the last condition is absent). The existence of a unique invariant distribution $\mu_\e\ll \mu $ for the perturbed process is obtained also in \cite{KO} and our expansion \eqref{expansiong} is equivalent to the expansion (4.5) in \cite{KO}, {see \cite[Appendix B]{ABFv2}  for more details.} In  Theorem \ref{teo_invariante} we have collected information on the exponential  convergence  of semigroups (which is relevant to get the invariance principle in  Proposition \ref{gelem_gelem}), while in \cite{KO}  the exponential  convergence of densities is derived. \end{remark}

 

\begin{remark}\label{wurstel}
 Let $h_\e$ be the Radon--Nykodim derivative of $\mu_\e$ w.r.t.\@ $\mu$.  Let  $A \subset \O$   be a Borel set  such that $\mu_\e (A)=0$. Since $ 0=\mu_\e(A)= \mu(A) + \mu( (h_\e-1) \mathds{1}_A)$, 
we have
 $ \mu(A) = \mu( (1-h_\e) \mathds{1}_A) \leq \|\mathds{1}-h_\e\| \mu(A)^{1/2}$. Hence, by Theorem \ref{teo_invariante}-\eqref{ii}
 \begin{equation}\label{chitarra} \mu_\e (A)=0 \; \Rightarrow \;\mu(A) \leq \e^2/(\g-\e)^2\,.
 \end{equation}
 This implies that any property that holds $\mu_\e$--a.s.\@ holds also $\mu$--a.s.\@ if $\mu \ll \mu_\e$ and anyway, in the general case, holds for all $\eta \in \O$ with exception of a set of $\mu$--measure bounded by $\e^2/(\g-\e)^2$.
\end{remark}

We now concentrate on additive functionals for the perturbed process.
As an immediate consequence of Birkhoff ergodic theorem, Theorem \ref{teo_invariante} and \eqref{chitarra} in Remark \ref{wurstel}, we get:
 \begin{CorollaryA}[Law of large numbers]\label{lacho_drom} Let   Assumption \ref{ferragosto} be satisfied, let $\e< \g$ and 
let  $f : \O \to \bbR$ be a measurable function, nonnegative  or in $L^1(\mu_\e)$ (e.g. bounded or in $L^2(\mu)$). Then
\begin{equation}\label{malaguena} \lim _{t \to \infty} \frac{1}{t} \int_0^t f(\eta_s)= \mu_\e(f)\,, \qquad \bbP^{(\e)}_{\eta}-{\rm a.s.}
\end{equation}
for $\mu_\e$--a.e.\@ $\eta$  (recall Remark \ref{wurstel}). \end{CorollaryA}

We conclude this general part with an invariance principle:
\begin{Proposition}[Invariance principle for additive functionals]\label{gelem_gelem}
Suppose that  $\O$ is a Polish space and that the
 perturbed  process on $\O$ is Feller. Let   Assumption \ref{ferragosto} be satisfied, let $\e < \g$ and let $f: \O \to \bbR$ be a function in $C_b(\O)$.
  Given $n \in \bbN$, define the process 
$$
 B^{(n)}_t(f):= \int_0^{nt}  \frac{ f(\eta_s) - \mu_\e(f) }{\sqrt{n} }ds \,, \qquad t \in \bbR_+\,.
$$ Then 
there exists a constant $\s^2 \geq 0$ such that 
under $\bbP^{(\e)}_{\mu _\e}$  the process
 $ \bigl( B^{(n)}_t \bigr)_{t \in \bbR_+} $ weakly converges to a Brownian motion with diffusion coefficient $\s^2$. 
 \end{Proposition}
  Proposition  \ref{gelem_gelem} is proved in Section \ref{limiti_perturbato}, where a characterization of $\s^2$ is given.


 
\section{Proof of Lemma \ref{ponte}}\label{golden_gate}
 We first note that the semigroup $\tilde{S}_\e (t)$ is well defined  since $C(\O)= C_b(\O)$ due to compactness. Let us  prove the lemma. 
We claim  that  $ \cD(\tilde L_\e) \subset \cD(L_\e)$ and that  $\tilde L_\e f = L_\e f$ for all  $f \in \cD(\tilde L_\e) $. To prove our claim  fix $f \in     \cD(\tilde L_\e)$. By definition of core, there exists $f_n \in \cC_\e$ with $f_n \stackrel{\|\cdot \|_\infty}{\to } f$  and $\tilde L_\e f_n \stackrel{\|\cdot \|_\infty}{\to }\tilde L_\e f$. The convergence holds also in $L^2(\mu)$, while by assumption $f_n \in \cC_\e \subset \cD(L_\e)$ and $\tilde L_\e f_n = L_\e f_n$. Using that the operator $L_\e$ is closed in $L^2(\mu)$ (being an infinitesimal generator), we get that necessarily $f  \in \cD(L_\e)$ and 
$\tilde L_\e f=  L_\e f$, thus proving our  claim.
Let again $f \in \cD( \tilde L_\e)$. Then (cf.\@ \cite[Lemma 1.3, Chapter 2]{EN}) $ \tilde S _\e(t) f \in  \cD( \tilde L_\e)$. By the above claim we get that $ \tilde S _\e(t) f \in  \cD(  L_\e)$ and $L_\e \tilde S _\e(t) f = \tilde L_\e \tilde S _\e(t) f $. Since (cf.\@ \cite[Lemma 1.3, Chapter 2]{EN})  ${ \lim _{\d \to  0} \frac{ \tilde S_\e(t+\d) 
- \tilde S_\e(t) }{\d} f  }=   \tilde L_\e \tilde S _\e(t) f=   L_\e \tilde S _\e(t) f $ in uniform norm, the same must hold in $L^2(\mu)$ (if $t=0$, the above limit has to be taken with $\d \downarrow 0$). Collecting the above observations we get that the function 
$\varphi(t) : [0,+\infty) \ni t \mapsto \tilde S_\e (t) f\in L^2(\mu)$ has values in $\cD(L_\e)$ and satisfies the Cauchy problem  $ \varphi'(t) =L_\e \varphi (t) $, $\varphi(0)= f$, where $\varphi'(0)$ has to be thought as right derivative.  Since also the function $\bar \varphi(t) : [0,+\infty) \ni t \mapsto  S_\e (t) f\in L^2(\mu)$ satisfies the same properties, by the uniqueness  of the solution of the Cauchy problem (cf.\@ \cite[end of page 483]{K}) we conclude that   
$ \tilde S_\e (t) f  =   S_\e (t) f $, i.e.\@ we get  \eqref{ghiacciolo} for $f \in \cD( \tilde L_\e)$. To extend \eqref{ghiacciolo} to any $f \in C(\O)$ its enough to take $f_n \in  \cD( \tilde L_\e) $ with $\| f-f_n\|_\infty \to 0$. Then also $\| f-f_n \| \to 0$. At this point it is enough to take the limit $n \to \infty$ in the identity $ \tilde S_\e (t) f _n =   S_\e (t) f _n$ and use that  $\tilde S_\e (t) $ is a bounded operator in $C(\O)$, while $   S_\e (t)$ is a bounded operator in $L^2(\mu)$.

\section{Preliminary estimates on Dyson--Philipps expansion}\label{primino}

In this section we prove Proposition \ref{spezzatino} and the bound in \eqref{convergencetomean}. 
Let us first state a simple remark (whose proof is omitted since standard) that will be frequently used: 
\begin{remark}\label{luna_piena}
Since $\mu$ is a stationary distribution for the unperturbed process and the Poincar\'e  inequality \eqref{banana} is satisfied,  we have that
(i) $S(t)f=f$ for all $t \geq 0$ iff $f$ is a constant function, 
(ii)  $0$ is a simple eigenvalue of $L$,
(iii) $\mu(S(t) f)= \mu(f)$ for any $f \in L^2(\mu) $. Moreover, since $L_\e$ is a Markov generator, it must be $\hat L_\e f=0  $ for $f$ constant. 
\end{remark} 
In the next proposition, by means of the Poincar\'e inequality, we improve known general bounds concerning the Dyson--Phillips expansion.
In what follows, given $f \in L^2(\mu)$, we abbreviate (recall \eqref{Sn}) :
\begin{equation}\label{gigino}
g_{n}(t) :=S_\e^{(n-1)}(t) f, \, \text{ for any }n \geq 1,
\end{equation}
so that the Dyson--Phillips expansion in equation \ref{DP} reads as 
\begin{equation}
S_\e (t)f= \sum _{n =1} ^\infty g_n (t)  \,, \qquad f \in L^2(\mu)\,.
\end{equation}

\begin{Proposition}\label{prop_pert}
For each $f \in L^2(\mu)$ and  $n\geq 1$   it holds
\begin{align}
& \| g_n(t) - \mu (g_n (t)) \| \leq e^{-\gamma t } \frac{( \e   t)^{n-1} }{(n-1)!} \| f- \mu(f) \|\,,\label{variancedecay}\\
& | \mu\bigl( \hat L_\e g_n(t)\bigr) | \leq \e   e^{-\gamma t } \frac{(\e   t)^{n-1} }{(n-1)!} \| f- \mu(f) \| \,,\label{caffettino}\\
& | \mu(g_{n+1}(t)) | \leq (\e/\g)^{n}\|f-\mu(f)\|\,.\label{meandecay}
\end{align}
Moreover, $\mu( g_1 (t) )= \mu(f)$ and,  for each $n\geq 1 $,    
\begin{equation}\label{pioggia}\lim_{t \to \infty} \mu ( g_{n+1}(t))=\int_0^\infty \mu\bigl( \hat L_\e g_{n}(s) \bigr) ds\,,\end{equation}
 the  integral
being  well posed   due to \eqref{caffettino}. More precisely, it holds
\begin{equation}\label{neve}
\bigl| \mu ( g_{n+1}(t)) - \int_0^\infty \mu\bigl( \hat L_\e g_{n}(s) \bigr) ds \bigr|
\leq \| f -\mu(f) \|  \int_t ^\infty \e e^{-\g s} \frac{(\e s)^{n} }{n!} ds \,.
\end{equation}
\end{Proposition}

\begin{proof} 
To prove  \eqref{variancedecay} we bound
\begin{align*}
& \|g_{n+1}(t)-\mu(g_{n+1}(t))\|=\bigl\|\int_0^t S(t-s)\hat L_\e g_n(s)ds-\mu\bigl(\int_0^tS(t-s)\hat L_\e g_n(s)ds\bigr)\bigr\|\\
&\leq  \int_0^t\bigl\|S(t-s)\hat L_\e g_n(s)-\mu\bigl(S(t-s)\hat L_\e g_n(s)\bigr)\bigr\|ds\\
&
\leq  \int_0^t e^{-\g (t-s)}\bigl\|\hat L_\e g_n(s)-\mu\bigl(\hat L_\e g_n(s)\bigr)\bigr\|ds \leq   \int_0^t e^{-\g (t-s)}\bigl\|\hat L_\e g_n(s)      \bigr\|ds\\
&= \int_0^t e^{-\g (t-s)}\bigl\|\hat L_\e\bigl(g_n(s)-\mu(g_n(s))\bigr)\bigr\|ds \leq  \int_0^t e^{-\g (t-s)}\bigl\|\hat L_\e\bigr\|\bigl\|g_n(s)-\mu(g_n(s))\bigr\|ds,
\end{align*}
where  
the second inequality follows from  Item (iii) in Remark \ref{luna_piena} and from the $L^2$--exponential decay \eqref{banana}, the third one  uses that $\|f-\mu(f)\|\leq \|f\|$
  for any $f \in L^2(\mu)$.  
 With this established, we can check \eqref{variancedecay} inductively, noticing that for $n=1$, the inequality is just a consequence of  
 the $L^2$--exponential decay \eqref{banana} and   Item (iii) in Remark \ref{luna_piena}.

To prove \eqref{caffettino}, by Remark \ref{luna_piena} we can bound
$ |\mu( \hat L_\e g_n(s) ) |$ by  $| \mu ( \hat L_\e (g_n(s)  -\mu(g_n(s) ) )| \leq \| \hat L_\e\| \|g_n(s)  -\mu(g_n(s) )\|$.
At this point the thesis follows from  \eqref{variancedecay}.

To prove  \eqref{meandecay} we write $\mu(g_{n+1}(t))$ as $\int_0^t\mu (\hat L_\e g_n(s)  )ds$.
 By \eqref{caffettino} the last integral can be bounded by $ \frac{\e ^n}{(n-1)!}\|f-\mu(f)\|\int_0^\infty e^{-\gamma s}s^{n-1}ds$, thus leading to \eqref{meandecay}.

The identity $\mu(g_1(t))=\mu(f)$ follows from  Remark  \ref{luna_piena}. As in the proof of \eqref{caffettino},
 $\int_t^\infty \bigl|\mu\bigl(\hat L_\e g_n(s)  \bigr)\bigr|ds\leq  \int_t^\infty ds \e \,e^{-\gamma s}\frac{(\e  s)^{n-1}}{(n-1)!}\|f-\mu(f)\| $, which goes to zero as $t \to \infty$. Hence, $\mu(g_{n+1}(t))$ has limit \eqref{pioggia}, which is finite, and also \eqref{neve} holds.
\end{proof}

We have now the tools to prove  some assertions of Section \ref{viva_segovia}:

\begin{proof}[Proof of Prop. \ref{spezzatino} and  \eqref{convergencetomean}] Due to \eqref{variancedecay} and \eqref{meandecay} we can bound the l.h.s. of \eqref{volare} by 
\[
 \|f-\mu(f)\|\bigl\{\sum_{n=k}^\infty e^{-\gamma t}\frac{(\g  t)^n}{n!} (\e/\g)^{n}  +\sum_{n=k}^\infty(\e/\g) ^{n}\bigr\} \leq  \|f-\mu(f)\|2\sum_{n=k}^\infty (\e/\g)^n\,,
\]
thus leading to \eqref{volare}.
 
Due to the Dyson--Phillips expansion,  we can bound $ \| S_\e(t) f - \mu( S_\e(t)f ) \|$ by $\sum_{n\geq 1}\|g_n(t)-\mu(g_n(t))\|$, and  \eqref{convergencetomean}  follows  immediately from    \eqref{variancedecay}. 
\end{proof}



\section{Proof of Theorem  \ref{teo_invariante}}\label{assoluto}

Let us denote by  $\G(f)$ the r.h.s. of \eqref{muinfinity_bis}. We first observe that by \eqref{caffettino} the integral and series in the r.h.s. of \eqref{muinfinity_bis} are absolutely convergent, hence $\G(f)$ is well defined.
Moreover,  always by \eqref{caffettino}, we get
$| \G(f) |\leq \bigl(\g/(\g-\e) \bigr) \|f\|$.

Due to the Dyson--Phillips expansion, it holds   $ \mu( S_\e (t) f) = \sum _{n \geq 1} \mu( g_n(t) ) $. Hence,  one easily derives \eqref{barbiere}  with  $\mu_\e(f)$ replaced by $\G(f)$  from \eqref{neve}.
As a byproduct with  \eqref{convergencetomean} proved at the end of Section \ref{primino}, we conclude that
\begin{equation}\label{stanca}
 \lim _{t \to \infty} \| S_\e(t) f - \G(f) \| = 0 \,, \qquad f \in L^2(\mu) \,.
 \end{equation}

\subsection{Proof of Item \eqref{a}}
Consider now the perturbed Markov process  with initial distribution $\nu$ as in Item  \eqref{a}  and call $\nu^{(t)}_\e $  its distribution at time $t$. Take $f \in C_b (\O)$. We claim that \begin{equation}\label{canzoncina}
\nu ^{(t)}_\e (f) = \mu \Big( \frac{d \nu}{d\mu}\bbE^{(\e)}_\cdot ( f(\eta_t) ) \Big)=\mu  \Big( \frac{d \nu}{d\mu}S_\e (t) f \Big)\underset{t\to \infty}{\longrightarrow} \G(f) \,, \qquad f \in C_b (\O)\,.
\end{equation}
(note that the first identity is trivial, while the second follows from \eqref{ghiacciolo}).
To this aim it is enough to prove this equivalent claim: for any diverging sequence $t_n \nearrow \infty$  there exists a subsequence $t_{n_k}$ such that $\mu  \Big( \frac{d \nu}{d\mu}[S_\e (t_{n_k}) f-\G(f)] \Big) \to 0$ as $k \to \infty$. Since $S_\e (t_n) f - \G(f)\to 0 $ in $L^2(\mu)$, there exists a subsequence $t_k$ such that  
$S_\e (t_{n_k}) f - \G(f)\to 0$ $\mu$--a.s.. Hence $|S_\e (t_{n_k}) f - \G(f)|$ is a function bounded by $(1+\g/(\g-\e))\|f\|_\infty$ (recall \eqref{ghiacciolo}) and converging to zero $\mu$--a.s.\@. The equivalent claim  then follows by the dominated convergence theorem.

We know that $\G:L^2(\mu) \to L^2(\mu)$ is a bounded linear operator. By Riesz representation theorem, there exists $h_\e \in L^2(\mu)$ such that $ \G(f) = \mu( h_\e f) $ for each $f \in L^2(\mu)$. We observe  that $h_\e \geq 0$ $\mu$--a.s.\@ since  $\G(f) \geq 0$ for any $f \in C_{b,+} (\O)$ (cf.\@ Lemma \ref{miele}-(ii)).
Let us define  the nonnegative measure $\mu_\e$ as $d \mu_\e = h_\e d \mu$. By \eqref{canzoncina} we conclude that  $\mu_\e ( \mathds{1})=1$, hence $\mu_\e$ is a probability measure. Using that 
 $\G(f)= \mu_\e(f)$,  by \eqref{canzoncina} we get  Item \eqref{a}.


 \subsection{Proof of Item \eqref{iii}}
 Since $\mu_\e(f)= \G(f)$,  by the definition of $\G(f)$ we get \eqref{muinfinity_bis}.  We have already proved  \eqref{convergencetomean}  at the end of Section \ref{primino}, while at the beginning of this section we have shown that \eqref{barbiere} holds with $\G(f)$ instead of $\mu_\e(f)$. Since {these} two values are indeed equal, we get \eqref{barbiere} and therefore Item \eqref{iii}.
{\eqref{ikea} and  \eqref{benedicte} are  a simple consequence of  \eqref{muinfinity_bis} and  \eqref{caffettino}}.

\subsection{Proof of Item \eqref{ii}}  By construction, $\mu_\e \ll \mu$ with Radon--Nikodym derivative $h_\e$.
By \eqref{stanca} and since $\G(f)=\mu( h_\e f)= \mu_\e(f)$ for any $f \in L^2(\mu)$, we have that $ \mu_\e (f)=\lim_{t \to \infty} \mu_\e( S_\e(t) f) $ for any
 $f \in C_b(\O)$. Taking $f:= S_\e(s)g$ and afterwards $f:= g$ and using the semigroup property $S_\e(t+s) g = S_\e (t) S_\e(s) g$, we conclude that $\mu_\e( S_\e (s) g) = \mu_\e(g)$ for any $g \in C_b (\O)$. By Assumption \ref{ferragosto} this implies that $\mu_\e \bigl( \bbE^{(\e)} _\cdot [g(\eta_s) ]\bigr )= \mu_\e (g)$ for any $g \in C_b (\O)$, hence the invariance of $\mu_\e$ for the perturbed Markov process.  The uniqueness assertion  follows from {Item \eqref{a}}.

To derive the expansion \eqref{expansiong}, note first that for $n \geq 0$, and any $f \in L^2(\mu)$, we have
\begin{equation}\label{pasqua}  \mu\bigl( \hat L_\e S_\e^{(n)}(s) f \bigr)= \mu
\bigl( \bigl( [S_\e^{(n)}(s)]^*  \hat L_\e^* \mathds{1} \bigr) f)=: \mu
\bigl( \bigl( H_\e^{(n+1)}(s)  \mathds{1} \bigr) f), \end{equation}
and the recursions in \eqref{defH} easily follow.
By \eqref{caffettino}  and \eqref{pasqua}, we then get $\| H_\e^{(n)}(s)  \mathds{1}   \| 
  \leq \e   e^{-\gamma s } \frac{(\e   s)^{n} }{n!} $. It then follows that the integrals and the series  in the r.h.s. of  \eqref{expansiong}  are absolutely convergent in $L^2(\mu)$, and therefore, by \eqref{muinfinity_bis} and  \eqref{pasqua}, the expansion \eqref{expansiong} holds.

From \eqref{expansiong} and the bound $\| H_\e^{(n)}(s)  \mathds{1}  \| 
  \leq 
\e   e^{-\gamma s } \frac{(\e   s)^{n} }{n!} $ we get $\|h_\e-\mathds{1}\|\leq\frac{ \e}{\g-\e}$.
 \subsection{Proof of Item \eqref{v}}
Some of the ideas are taken from \cite{KO}[Sec. 3.1.2] although we show that   some assumptions there can indeed be avoided.

We know that $\mu_\e \ll \mu$ (see Item \eqref{i}).  We call  $A_\e$ the $\mu$--support of  ${ h_\e=}d \mu_\e / d\mu$. We only need to prove that  $\mu( A_\e^c)=0$. By stationarity of $\mu_\e$ w.r.t.\@ the perturbed dynamics, we have
\begin{equation}
0 = \mu_\e (A_\e ^c) = \mu_\e \left( S_\e (t) \mathds{1} _{ A_\e^c}
\right)=  \int \mu(d \eta){h_\e} (\eta) \bbP_\eta ^{(\e)} \left[ \eta_t \in A_\e^c \right] \,.
\end{equation}
Hence, $\mu \bigl( \{ \eta \in A_\e \,:\, \bbP_\eta ^{(\e)} \bigl[ \eta_t \in A_\e^c \bigr]>0 \} \bigr) =0$. 
By condition \eqref{ho_fame} we conclude that  $\mu \bigl( \{ \eta \in A_\e  \,:\, \bbP_\eta  \bigl[ \eta_t \in A_\e^c \bigr]>0 \} \bigr) =0$. This implies that  the function $\eta \mapsto 
 \bbP_\eta  \bigl[ \eta_t \in A_\e^c \bigr] = S(t) \mathds{1} _{A^c_\e}(\eta) \in [0,1]$ is zero on $A_\e$ $\mu$--a.s., hence $ 0 \leq  S(t) \mathds{1} _{A^c_\e} \leq  \mathds{1} _{A^c_\e}$ $\mu$--a.s..
 Suppose by absurd that $\mu ( \{ \eta \,:\, S(t) \mathds{1} _{A^c_\e} ( \eta) < 
\mathds{1} _{A^c_\e} ( \eta) \} )>0$. Then we would conclude that $\mu (S(t) \mathds{1} _{A^c_\e} ) < \mu ( \mathds{1} _{A^c_\e})$, in contradiction with the stationarity of $\mu$ w.r.t.\@ $S(t)$. Hence, it must be $ S(t) \mathds{1} _{A^c_\e} =  \mathds{1} _{A^c_\e}$ $\mu$--a.s.. Since this holds for each $t $, by the ergodicity of $\mu$ we conclude that $\mu ( A^c _\e)\in \{0,1\}$. If $  \mu ( A^c _\e)=1$, then ${h_\e}\equiv 0$ $\mu$--a.s., while $\mu({h_\e})=1$. It remains the case $\mu ( A^c _\e)=0$, which 
implies that $\mu \ll  \mu_\e$.


\subsection{Proof of Item \eqref{iv} } Let $C_{b,+} (\O):= \{ f \in C_b(\O) \,:\, f\geq 0\}$ and $L^2 _+ (\O):= \{f \in L^2 (\mu) \,:\, f \geq 0 \; \mu\text{--a.s.}\}$.  
By Assumption \ref{ferragosto} we have $	{S_\e(t)} f \geq 0$ $\mu$--a.s.\@ for any $f \in C_{b,+} (\O)$.
Since $C_{b,+}(\O)$ is $\|\cdot\|$--dense in $L^2 _+ (\O)$, as immediate consequence of Lemma \ref{miele}-(i), we conclude that   ${S_\e (t) } f \geq 0$ $\mu$--a.s.\@ for any $f \in L^2_+(\mu)$. Since $\|f\|_\infty -f \in L^2_+(\mu)$  and ${S_\e(t)}  \|f\|_\infty = \|f\|_\infty$, we then conclude that ${S_\e(t)} f \leq \|f\|_\infty$ $\mu$--a.s.\@ for any $f \in  L^\infty (\mu) $.  By applying the last bound to  $-f$, we get $|{S_\e(t)} f | \leq \|f\|_\infty$ $\mu$--a.s.\@ for any $f \in L^\infty (\mu) $.  

The above considerations and Schwarz inequality imply for   any  $f \in L^\infty (\mu) $ that 
\begin{equation}
\begin{split}
\| S_\e (t) f -\mu_\e(f) \|_\e ^2 & = 
\mu_\e ( |S_\e (t) f -\mu_\e(f) |^2)\leq  \|f- \mu_\e(f) \|_\infty  \mu ( h_\e |S_\e (t) f -\mu_\e(f) |)\\
& \leq 
 \|f- \mu_\e(f) \|_\infty   \| h_\e\| \cdot  \|S_\e (t) f -\mu_\e(f) \|\,.
\end{split}
\end{equation}
By \eqref{benedicte} in Item \eqref{iii} we have $ \|f- \mu_\e(f) \|_\infty \leq  \|f- \mu (f) \|_\infty  +|\mu(f)- \mu_\e(f) | \leq \frac{\g}{\g-\e}  \|f- \mu (f) \|_\infty$.
By Item \eqref{ii} we have  $\|h_\e \| \leq \g /(\g-\e)$ and by Item \eqref{iii} we have 
$ \|S_\e (t) f -\mu_\e(f) \|\leq [ \g /(\g-\e)] e^{- (\g-\e)t} \|f-\mu(f)\|$. Hence the conclusion.


\subsection{Proof of Item \eqref{i} } Recall Remark \ref{compleanno}.
  Let $B \subset \O$ be a Borel set satisfying $\mathds{1}_B (\eta_t) =\mathds{1}_B (\eta_0)$ $\bbP^{(\e)}_{\mu_\e}$--a.s.\@ {for all $t\geq 0$.} By Lemma \ref{candelina}\footnote{In the proof of Lemma \ref{candelina} we use Theorem \ref{teo_invariante} but not Item (vi).} we then have $S_\e (t) \mathds{1}_B = \mathds{1}_B$ $\mu_\e$--a.s.. {Then, by \eqref{arachidi_bis}, we conclude that $\mathds{1}_B =\mu_\e(B)$  $\mu_\e$--a.s., thus implying that $\mu_\e(B) \in \{0,1\}$.}




\section{Proof of the invariance principle in  Proposition \ref{gelem_gelem}}\label{limiti_perturbato}

Given $ h \in L^2(\mu_\e)$, we introduce the functional  $ A_t( h)= \int_0^t h(\eta_s) ds $   defined on the path space $D(\bbR_+, \O)$.
By Schwarz inequality and stationarity of $\mu_\e$ for the perturbed process,  we can bound
\begin{equation}\label{succhino}
\|  A_t(h)  \|_{  L^2 \bigl(\bbP^{(\e)} _{\mu_\e}\bigr)}= \bbE^{(\e)}_{\mu_\e}\left[ A_t(h)^2 \right]  ^{1/2}
\leq   t \| h \|_\e \,.
\end{equation}
  The  family of  operators $  \mathcal{S}_\e(t)h (\eta):= \bbE^{(\e)}_\eta \bigl[ h(\eta_t) \bigr]$, $t \in \bbR_+$, is   a  well defined   strongly continuous contraction semigroup 
in $L^2(\mu_\e)$ for $t \in \bbR_+$ (see Lemma \ref{SC} and its proof). We write $\mathcal{L}_\e: \cD(\mathcal{L}_\e) \subset L^2(\mu_\e) \to L^2(\mu_\e)$ for its infinitesimal generator. Do not confuse the above operators $\mathcal{S}_\e(t), \mathcal{L}_\e$ with the previously defined $S_\e(t), L_\e$ which live in $L^2(\mu)$. On the other hand, by \eqref{ghiacciolo}, given $h \in  C_b(\O) $ it holds $\mathcal{S}_\e(t)h = S_\e(t) h$ $\mu$--a.s.\@  and therefore $\mu_\e$--a.s.\@ (since $\mu_\e \ll \mu$). 

 Let $f\in C_b(\O) $,  as in the theorem. 
Since along the proof $\e$ is fixed,  at  cost of replacing $f$ by $f-\mu_\e(f)$ we assume that  $\mu_\e(f)=0$. 
Due to \eqref{arachidi_bis} and the previous observations, we can bound 
\begin{equation}\label{bisogno}
 \int _0^\infty \| \mathcal{S}_\e(t)  f dt \| _\e 
 =
  \int _0^\infty \| S_\e(t)  f dt \| _\e =: \k < \infty\,.
  \end{equation}

Hence,  $ g:=\int _0^\infty \mathcal{S}_\e(t)  f dt$ is a  well defined element  of  $L^2(\mu_\e)$.  Since $\mathcal{S}_\e(r) g- g=-\int_0 ^r \cS_\e(t)  f dt$ and since $\mathcal{S}_\e(t)  f \to f$ in $L^2(\mu_\e)$ as $t \downarrow 0$, by definition of infinitesimal generator  we get that $g \in \cD( \mathcal{L}_\e)$ and $-\mathcal{L}_\e g =f$. 

As a consequence  we can write 
$A_t(f)= M_{t}  + R_{t}$,
where
\begin{align}
& M_{t}:= g (\eta_t)- g  (\eta_0) - \int_0^t \mathcal{L}_\e g ( \eta_s) ds\,,\label{Din1}\\
& R_{t }:=- g ( \eta_t)+ g( \eta_0) \,.\label{Resto}
\end{align}
By \eqref{bisogno}, we get that 
\begin{equation}\label{sirenetta}
\|R_t \|_{  L^2 (\bbP^{(\e)} _{\mu_\e})}  \leq 2 \k \,.
\end{equation}

In what follows,  we   apply the invariance principle for martingales as stated in \cite[Thm. 2.29, Chp. 2]{KLO}, which holds for c\`adl\`ag martingales w.r.t.\@ filtrations satisfying the usual conditions. To this aim,  
 we  take 
the augmented filtration $( \bar \cF_t)_{t \geq 0}$ w.r.t.\@ $\bbP^{(\e)}_{\mu_\e}$ of the natural filtration $( \cF_t)_{t \geq 0}$, where $\cF_t:= \s ( \eta_s : 0 \leq s \leq t)$  \cite[Chp. 2]{KS}. Since we have assumed that the perturbed process is Feller, then this filtration satisfies the usual condition w.r.t.\@ $\bbP^{(\e)}_{\mu_\e}$ \cite{KS}[Prop. 7.7, Chp. 2]. It is known (cf.\@ \cite[Chp. 2]{KLO} that  $(M_t) _{t \geq 0}$ is  a martingale  w.r.t.\@ the augmented filtration $( \bar \cF_t)_{t \geq 0}$. Below  we work with   the c\`adl\`ag modification of $(M_t)_{t \geq 0}$ (cf.\@ \cite[Thm. 3.13, Chp. 1]{KS}), that we still call $M_t$ with some abuse of notation.

We split the rest of the proof in two parts. First we show an invariance principle for the martingale  $M_t$, afterwards we prove that the rest $R_t$ is negligible (cf.\@ Lemma \ref{lufthansa} and Lemma \ref{monnezza}). 

%

\begin{Lemma}\label{lufthansa} 
For any $t\geq 0$, define $ M_t^{(n)}:= \frac{M_{nt}}{\sqrt{n}} , n\in\bbN$.
Then,  under $\bbP ^{(\e)}_{\mu_\e}$,
the rescaled process $( M^{(n)}_t )_{t \in \bbR_+}$ weakly converges to a Brownian motion with diffusion constant $\sigma^2:=\bbE^{(\e)}_{\mu_\e}(M_1^2)\geq 0$.
\end{Lemma}

\begin{proof} The martingale $M_t$ is square integrable w.r.t.\@ $\bbP^{(\e)}_{\mu_\e}$. According to \cite[Thm. 2.29, Chp. 2]{KLO} we only need to prove that 
$\langle M \rangle _k /k$ converges to $\bbE^{(\e)}_{\mu_\e}(M_1^2)$ both $\bbP^{(\e)}_{\mu_\e}$--a.s.\@ and in $L^1( \bbP^{(\e)}_{\mu_\e})$. 
To this aim we write $ \langle M \rangle _k = \sum _{j=0}^{k-1} (\langle M\rangle_{j+1} - \langle M\rangle_{j})$ and we point out (cf.\@ \cite[chp.2]{KLO}) that  
$$
 \langle M\rangle_{j+1} - \langle M\rangle_{j}= \langle M \rangle_1 \circ \theta _j \, \quad \bbP^{(\e)}_{\mu_\e}\text{--a.s.}.
$$
Moreover, we have     $\bbE^{(\e)}_{\mu_\e}  \bigl( \langle M\rangle_1 \bigr) = \bbE^{(\e)}_{\mu_\e}  \bigl( M^2_1 \bigr)< \infty$.

At this point the thesis follows from the a.s and the $L^1$--Birkhoff ergodic theorem and the ergodicity of $\bbP^{(\e)}_{\mu _\e}$ w.r.t.\@ time translations.
\end{proof}

In the following lemma we give a bound to control the trajectories of the  error $R_t$ (cf.\@ \eqref{Resto}).  
The proof of the lemma is based on a simple block-decomposition argument in the spirit of \cite{RAS}.

\begin{Lemma}\label{monnezza}We have   
\begin{equation}\label{riciclo}
\lim _{T \to \infty }\frac{\verde{\sup}_{ t \leq T} |R_t| }{\sqrt{T} }=0 \qquad \text{ in $\;\;\bbP^{(\e)}_{\mu_\e}$--probability  }\,.
\end{equation}
\end{Lemma}

Since,  as shown in \eqref{Din1},  we can write $ B^{(n)}_t(f)= M_t^{(n)}+ \frac{R_{tn}}{\sqrt{n} }$, Proposition \ref{gelem_gelem} follows  from Lemma \ref{lufthansa} and Lemma \ref{monnezza}.

\begin{proof}[Proof of Lemma \ref{monnezza}]  Given a positive integer $j$, let 
$m_j:= \lceil j^{4/5}\rceil$ and  $\ell_j= \lceil j^{1/5}\rceil$. 
Consider the partition of the time interval $[0,j]$ in $m_j$ sub--intervals $I^j_k:= [k l_j,(k+1)l_j)$, $k=0,\ldots,m_j-1$, of {measure} $\ell_j$ (for simplicity of notation we assume that $m_j \ell_j= j$).
From the decomposition $
A_t(f)= M_{t}  + R_{t}$  we can bound 
\begin{equation}\label{lyra}
\begin{split}
\sup_{ t\leq  j} |R_t | &\leq \max_{ k=0,1, \dots , m_j} | R_{k \ell_j} |  
+ \max _{  k= 0,1, \dots ,m_j-1 } \, \sup_{u\in I^j_k} |A_u(f)-A_{k \ell_j}(f) |
 \\
& + \max_{ k= 0,1, \dots ,m_j-1}\,  \sup_{u\in I^j_k} |M_u -M_{k \ell_j} |
 =: C_{1,j}+ C_{2,j} + C_{3,j}
\end{split}
\end{equation}

$\bullet $ {\bf Step 1}.
We first control the term
\[C_{3,j}:= \max_{ k= 0,1, \dots ,m_j-1}\,  \sup_{u\in I^j_k} |M_u -M_{k \ell_j} | \,.\]
Due to Lemma \ref{lufthansa} \verde{and the fact that for standard Brownian motion $W_t$, and any $C,T>0$, it holds $P(\sup_{t\leq T} |W_t|\geq C)\leq \exp\{-C^2/2T\}$}, we have $\lim _{j \to \infty} \bbP^{(\e)}_{\mu_\e} ( C_{3,j} > \d \sqrt{j})=0$.  

$\bullet $ {\bf Step 2}.
\verde{Let us now consider the term
$C_{1,j}:=  \max_{ k=0,1, \dots , m_j} | R_{k \ell_j} |$.
By a union bound, Markov inequality  and the uniform bound in \eqref{sirenetta}, for any $\d>0$, we can estimate}
\begin{equation}
\bbP^{(\e)} _{\mu_\e} \left(  \max_{ k=0,1, \dots , m_j} | R_{k \ell_j} |  \geq  \d \sqrt{j}  \right) \leq
\frac{4\k^2  m_j}{\d^2 j}=\cO(j^{-1/5})\,,
\end{equation}


$\bullet $ {\bf Step 3.}
\verde{We control the remaining term 
$$C_{2,j}:=\max_{  k= 0,1, \dots ,m_j-1 } \, \sup_{u\in I^j_k} |A_u(f)-A_{k \ell_j}(f) | \,.$$
First we observe that
$$| A_u(f)-A_{k \ell_j}(f) | \leq \int_{I^j_k}|f(\eta_s) |ds \,, \qquad \forall u \in I^j_k,\, k\leq m_j-1.$$ 
Thus the event $\{ C_{2,j} \geq \d\sqrt{j} \}$ is contained in the event  $B_j$ defined as
$$ B_j := \Big\{ \max _{  k= 0,1, \dots ,m_j-1 }\int_{I^j_k} |f(\eta_s) |ds  \geq \d \sqrt{j} \Big\}\,.$$
Using a union bound and stationarity, we have
\[
\bbP^{(\e)} _{\mu_\e} ( B_j) \leq m _j \bbP^{(\e)} _{\mu_\e} \Big( \int_{I^j_0} |f(\eta_s) |ds  \geq \d\sqrt{j}  \Big)\,.
\]}
\verde{Since $f$ is bounded, we can write
\begin{equation}
\begin{split}
\bbP^{(\e)} _{\mu_\e} ( B_j) \leq \frac{m _j}{ \d^3 j^{3/2}}   \bbE_{\mu_\e}^{(\e)} \left[  \left( \int _{0} ^{ \ell_j} |f(\eta_s) |ds \right)^3 \right]\leq 
c(f) \frac{ m_j \ell_j^{3}}{\d^3 j^{3/2}}=\cO(j^{-1/10}),
\end{split}
\end{equation}
with  $c(f)$ being a positive constant depending only on $f$.}
%
\medskip

We can now conclude the proof. 
 We consider a generic $T\geq 1$ and let $j$ be such that $ j \leq T < {j+1}  $. Since 
\[
\frac{ \sup_{t \leq T} |R_t|}{\sqrt{T} } \leq \frac{ \sup_{t \leq {j+1}} |R_t|
}{ \sqrt{{j+1} } } \frac{ \sqrt{{j+1} }}{\sqrt{j }} \,.
\]
from  \eqref{lyra}, the arbitrariness of $\d$ together with the three steps above, we get  
\[\lim _{T \to \infty }\bbP^{(\e)}_{\mu_\e}\Big(  T^{-1/2} \sup_{t \leq T} |R_t| \geq   \d  \Big) =0\] for any $\d>0$, and therefore the thesis.
\end{proof}

\begin{remark} 
The above proof is an extension to  the non--reversible case of the classic Kipnis--Varadhan approach. Lemma \ref{lufthansa} is standard, but the control of the rest $R_t$ provided in Lemma \ref{monnezza} does not follow from the estimates in \cite{KV} (note in particular that Lemma 1.4 there requires reversibility).  We mention that an  alternative strategy in the  non--reversible setting is  given by \cite[Thm. 2.32, Chp. 2]{KLO}, which on the other hand would require additional assumptions on the Markov process.
\end{remark}


\section{A coupling}\label{san_valentino}
In this section we describe a coupling between the dynamic random environment, the unperturbed random walk and the perturbed random walk.
To this aim we  define 
  $\l:= \sup _{\eta } \sum_{y \in \bbZ^d } \bigl(r(y, \eta)+ \max\{ 0,  \hat r_\e (y, \eta)\} \bigr)$,   which is finite due to { \eqref{nthMoment} in Assumption  \ref{ass:rates}}.  
  For each $\eta \in \O$ we fix two  partitions
  $$ [0,1]= \left(\cup _{y \in \bbZ^d }I(y,\eta) \right) \cup J(\eta)\,, \qquad [0,1]= \left(\cup _{y \in \bbZ^d }I_\e(y,\eta) \right) \cup J_\e(\eta) \,,$$
  where $I(y,\eta)$, $J(\eta)$, $I_\e(y, \eta)$, $J_\e(\eta)$ are {Borel sets} such that 
  \[ |I(y, \eta)|= r(y,\eta)/\l\,, \qquad | I_\e (y,\eta)| = r_\e(y,\eta)/\l\,,\]
  \[ | I(y,\eta) \cap I_\e(y, \eta)|= \bigl[ r(y,\eta)+  \min\{0, \hat  r_\e (y,\eta)\}\bigr] /\l\]
 (above $| I | $ denotes the {measure} of the {set} $I$). The above partitions are chosen with the property that the characteristic function   $(a, \eta) \mapsto  \mathds{1} \left( a \in I(y, \eta  ) \right )$ is measurable for any $y \in \bbZ^d$, where $(a, \eta) \in [0,1] \times \O$. The same must be valid for $I_\e (y, \eta)$. 

 \smallskip
 
 Let $\bbP_\eta^{\rm env}$ be the law of the  dynamic random environment, i.e.\@ the  process with generator $L_{\rm env}$ starting at $\eta$. We denote by $(\s_t)_{t \in \bbR_+}$ a generic trajectory of this process.    We  build  a Poisson point process $\cT:=\{t_1<t_2<\cdots \}\subset \bbR_+$   with intensity $\l$  on a suitable probability space with probability measure $P^{\rm Poisson}$.  We then  build  a sequence  $\cU:= (U_k)_{k\geq 1}$  of i.i.d. uniform variables taking value in $[0,1]$ on another probability space with probability measure $P^{\rm uniform}$. We then consider the product probability space with probability measure
 $ \cP_\eta := \bbP_\eta^{\rm env} \otimes P^{\rm Poisson} \otimes P^{\rm uniform}$. 
 
 \smallskip
 
We now consider the function $ F \left( (\s_t)_{t \in \bbR_+}, \cT, \cU\right)$ with value in $D(\bbR_+; \bbZ^d)$ 
associating with $ (\s_t)_{t \in \bbR_+}$, $\cT=\{t_1<t_2<\cdots \}$,  $\cU= (U_k)_{k\geq 1}$ the path $(x_t)_{t \in \bbR_+}$ defined as follows. We set $x_s = 0$ for all $s \in [0,t_1)$. Suppose in general that $x_t$ has been defined for any $t \in [0, t_k)$, with jump times $t_1,t_2 \dots, t_{k-1}$.  Set $z:= x_{t_k-}$ and  $\zeta:= \s_{t_k-}$. If 
 $ U_k \in I(y,\t_z \zeta)$ for some $y\in \bbZ^d$, then we set $x_t:=z + y$ for any $t \in [t_k, t_{k+1})$,  otherwise set $x_t:=z $ for any $t \in [t_k, t_{k+1})$. Since $\lim _{n \to \infty}t_n=\infty$ $\cP_\eta$--a.s., the definition of  $F$  is well posed $\cP_\eta$--a.s.. By construction, sampling  $\left( (\s_t)_{t \in \bbR_+}, \cT, \cU\right)$
according to $\cP_\eta$, the random  path $ \left(   (\s_t)_{t \in \bbR_+}, F \left( (\s_t)_{t \in \bbR_+}, \cT, \cU\right)\right)$ is the joint Markov  process  given by the dynamic random environment and the unperturbed random walk.  In particular,  sampling  $\left( (\s_t)_{t \in \bbR_+}, \cT, \cU\right)$
according to $\cP_\eta$, the random  path $\t_{F(t) } \s_t$ has law $\bbP_{\eta}$, where $F(t)$ stands for the process $F \left( (\s_t)_{t \in \bbR_+}, \cT, \cU\right)$ evaluated at time $t$. If in the above definitions, we replace ``$ U_k \in I(y,\t_z \zeta)$'' by ``$ U_k \in I_\e(y,\t_z \zeta)$'', we get a new function $F_\e$ and,  sampling  $\left( (\s_t)_{t \in \bbR_+}, \cT, \cU\right)$
according to $\cP_\eta$, the random  path $\t_{F_{\e}(t) } \s_t$ has law $\bbP^{(\e)}_{\eta}$.

In the sequel, we denote   by  $\mathcal{E}_{\eta}$ the  expectation corresponding to $\cP_\eta$, and we adopt the convention that $(X_t)_{t\geq 0}:=F \left( (\s_t)_{t \in \bbR_+}, \cT, \cU\right)$ denotes the walker process in the unperturbed setting,   while $(X_t^{(\e)})_{t\geq 0}:= F_\e \left( (\s_t)_{t \in \bbR_+}, \cT, \cU\right)$ refers to the walker process  in the  perturbed setting.
Given $\nu$ probability measure on $\O$ we define  $\cP_\nu:= \int \nu(d \eta) \cP_\eta$ and we write $ \cE_\nu$ for the expectation w.r.t.\@ $\cP_\nu$.


\section{Proof of Theorem \ref{late} (asymptotic stationary state and velocity)}\label{RWproof_mela}

\subsection{{Connection with $L^2$--perturbation of Markov processes discussed in Section \ref{viva_segovia}}}\label{auto}
 {The operator $L_{\rm env}: \cD(L_{\rm env}) \subset L^2(\mu) \to L^2 (\mu)$
 is the closure in $L^2(\mu)$ of the Markov generator $\cL_{\rm env}: \cD( \cL_{\rm env}) \subset C(\O) \to C(\O)$, shortly $\bigl( L_{\rm env},\cD(L_{\rm env})
\bigr)= \overline{  \bigl(  \cL_{\rm env}, \cD( \cL_{\rm env}) \bigr) }$ 
  (see e.g. the proof of  \cite[Prop.4.1, Chp.IV]{L}).  Recall the definition \eqref{RateTurbati}  of the operator $ \hat L_\e$ and set   $L_{\rm jumps} f  := \sum_{y \in \bbZ^d}r(y,\eta)\big[f(\tau_{y}\eta)-f(\eta)\big] $. As done for $\hat L_\e$, one can easily prove that $L_{\rm jump} $ is a bounded operator on $L^2(\mu)$.}
  
{Due to the previous observations, the $L^2(\mu)$--generator $L_{\rm ew}$ of the environment seen by the unperturbed walker is the closure of  the associated $C(\O)$--generator  $ \cL_{\rm env}+ \cL_{\rm jump}$ with domain $\cD( \cL_{\rm env})$.
 In particular, we have 
\begin{equation}\begin{aligned}\label{EnvFromURW}
{L_{\rm ew} f = L_{\rm env}f +L_{\rm jump}f }\,,\qquad \qquad  f \in \cD( L_{\rm ew})= \cD( L_{\rm env}) \, .
\end{aligned}\end{equation}
We introduce  the  operator   $L^{(\e)}_{\rm ew}: \cD( L^{(\e) } ) \subset L^2(\mu) \to L^2(\mu)$defined as 
\begin{equation}\label{Decomposition}L^{(\e)}_{\rm ew}:=L_{\rm ew}+\hat L_\e\,, \qquad \cD( L^{(\e)}_{\rm ew}):= \cD( L_{\rm ew})= \cD( L_{\rm env})\,. 
 \end{equation}}
{As already observed $\bigl( L_{\rm env},\cD(L_{\rm env})
\bigr)$ is the closure  in $L^2(\mu)$ of $ \bigl(  \cL_{\rm env}, \cD( \cL_{\rm env}) \bigr) $. From this property it is simple to check that $(L^{(\e)}_{\rm ew} , \cD(  L^{(\e)}_{\rm ew} ) )$ is the closure in $L^2(\mu)$ of the $C(\O)$--generator  $(\cL_{\rm env}+ \cL_{\rm jump}+\hat{\cL}^{(\e)}, \cD( \cL_{\rm env}) )$ of the  Feller process given by the environment seen by the perturbed walker (cf.\@ Assumption \ref{ass:feller}). Recall that  $S_\e(t)$ denotes  the semigroup in $L^2(\mu)$ generated by $ L_{\rm ew}^{(\e)}$. By applying Lemma \ref{ponte} with $\cC_\e:= \cD( \cL_{\rm env})$, we get that  $S_\e(t)$ satisfies the identity \eqref{ghiacciolo}.   }

\smallskip

Due to the  following proposition, we are in the setting of Section \ref{viva_segovia}. {Indeed, the unperturbed Markov process in Section \ref{viva_segovia} is the environment viewed from the unperturbed walker  $X_t$, 
 and $L^{(\e)}_{\rm ew}$ can be thought of  as the perturbed form of
$L_{\rm ew}$:}
\begin{Proposition}\label{chiave101}
Assumption \ref{ferragosto} is satisfied when the operators $L^{(\e)}_{\rm ew}$ and $L_{\rm ew}$ play the role of $L_\e$ and $L$  in \eqref{turbato},  respectively.
\end{Proposition}
\begin{proof}
By Assumption \ref{ass:stat}, 
$\mu $ is stationary for the environment seen by the unperturbed walker. For what concerns the Poincar\'e inequality, we observe that for any  $f \in \cD( L_{\rm ew})$ it holds
\begin{equation}\label{gap2} -(f,L_{\rm ew}f)_\mu =- (f,L_{\rm env} f)_\mu-(f, L_{\rm jump} f)_\mu\geq - (f,L_{\rm env} f) _\mu \geq \g \var_\mu(f)\,, \end{equation}
since 
\[-(f,L_{\rm jumps}f)_\mu= \frac{1}{2}\sum _{y } \int \mu(d\eta)\frac{r(y, \eta)+ r(-y, \t_y \eta)}{2}  \left[ f( \t_y \eta)- f (\eta) \right]^2 \geq 0 \,.\]
Due to Remark \ref{compleanno}, $\mu$ is ergodic for the unperturbed  environment viewed from the walker. {Finally, as already pointed out, identity \eqref{ghiacciolo} is satisfied.}
\end{proof}

Having Theorem \ref{teo_invariante}, Proposition \ref{chiave101} and  Proposition \ref{prop_pert},  Items (i) and (iv) of   Theorem \ref{late} become trivial  (for Item (iv) apply in particular \eqref{muinfinity_bis} and \eqref{caffettino}). Below we prove     Items (ii) and (iii). 
\subsection{Proof of Theorem \ref{late}--(ii)} By Item (i)  we know that  $\mu_\e \ll \mu$. We prove that $\mu \ll \mu_\e$ by means of 
the  criterion  given in  Theorem \ref{teo_invariante}--(\ref{v}). Fix $\eta \in \O$, $t>0$ and $B \subset \O$ measurable.  Recall (cf. Section~\ref{viva_segovia}) that $ \bbP_\eta^{(\e)}$ [$ \bbP_\eta$] is the law   of the environment viewed from the perturbed [unperturbed]  walker, and $\bbP_\eta^{\rm env}$ is the law of the dynamic environment. 
Given $\underline t=(t_1, t_2, \dots, t_k)$ with $0 < t_1<t_2 < \cdots < t_k \leq t$ and $\underline y= (y_1,y_2, \dots, y_k)$ we set
\[ A_\e(\underline{t}, \underline{y}):=\bbE^{\rm env} _\eta\Big[ \mathds{1}( \t_{y_1+\cdots +y_k} \s_t \in B) \cdot \prod _{i=1}^k r_\e( y_i , \t_{y_1+y_2+ \cdots + y_{i-1}} \s_{t_i} )  \Big]
\]
and we define $A(\underline{t}, \underline{y})$ similarly, with 
$r(\cdot , \cdot)$ instead of $r_\e(\cdot, \cdot)$.

By the construction of the process given in Section \ref{san_valentino} we have
\begin{equation}\label{noi}
\bbP_\eta^{(\e)} ( \eta_t \in B) = \sum_{k=0}^\infty \frac{e^{-\l t} t^k}{k!}
\sum_{y_1,y_2, \dots, y_k } \int_0^t dt_1 \int_{t_1}^t dt_2 \cdots \int_{t_{k-1} } ^t dt_k   A_\e(\underline{t}, \underline{y}) \,,
\end{equation}
and a similar formula relates $\bbP_\eta  ( \eta_t \in B) $ to $A(\underline{t}, \underline{y})$.

We first assume that \eqref{ellipticity} is satisfied. 
  If $\bbP_\eta^{(\e)} ( \eta_t \in B) =0$, then by \eqref{noi}  $A_\e(\underline{t}, \underline{y})=0$ for almost every choice of $t_1, \dots, t_k$ and for each choice of $y_1,\dots, y_k$. Then the same holds for $A(\underline{t}, \underline{y})$ and by the analogous of \eqref{noi} in the  unperturbed case we conclude that 
$\bbP_\eta ( \eta_t \in B) =0$. Hence the criterion given by Theorem \ref{teo_invariante}--(\ref{v}) is satisfied and $ \mu \ll \mu_\e$.

Let us suppose now that there are subsets $V, V_\e$ satisfying properties (a), (b), (c) in Item (ii). 
If  $\bbP_\eta ( \eta_t \in B) >0$, by the analogous of \eqref{noi} in the unperturbed case we conclude that there exist  $\underline t=(t_1, t_2, \dots, t_k)$ and $\underline y= (y_1,y_2, \dots, y_k)$ such that 
 $A(\underline  t, \underline  y)>0$. In particular,  $y_1, y_2, \dots, y_k \in V$ and  $\bbP_\eta^{\rm env} ( \t_y \s_t \in B)>0$ where $y=y_1+\cdots +y_k$.  For simplicity of notation we take $k=1$ (the argument can be easily generalized). By condition (c) we can write $y_1= z_1+ \dots+ z_r$ with $z_i \in V_\e$. By condition (b) and   since $\bbP_\eta^{\rm env} ( \t_y \s_t \in B)>0$,  we conclude that 
 $ A_\e ( (t_1, \dots, t_r), (z_1, \dots, z_r) )>0$ for each choice of $(t_1, t_2, \dots, t_r)$. 
 This together with \eqref{noi} implies that $\bbP_\eta^{(\e)} ( \eta_t \in B) >0$, hence by
 Theorem \ref{teo_invariante}--(\ref{v})  we get $ \mu \ll \mu_\e$.

\subsection{Proof of Theorem \ref{late}--(iii)}\label{lavatrice} We refer to the construction of the random walk $(X_t^{(\e)} )_{t \geq 0}$ given in Section \ref{san_valentino}. It is convenient here to identify the  Poisson point process $\cT= \{ t_1 < t_2 < t_3< \cdots \}$   with the Poisson process $N=(N_t)_{t \geq 0}$ having $\cT$ as set of jump times. {Moreover,  also for later uses, it is convenient to enlarge the random sequence $(U_k)_{k \geq 1}$ by adding $U_0$, with $(U_k)_{k \geq 0}$ i.i.d.. Further, we define the function $V: \bbR_+ \to [0,1]$ by setting $V_s:= U_k$ for {$ s\in [t_k,t_{k+1}) $}, with the convention $t_0:=0$. 
 Without loss of generality  we can assume that the product  probability measure $\cP_\eta$ introduced in Section \ref{san_valentino} (with the modification due to $U_0$) is defined  directly on the product  measure  space $\Theta:= D( \bbR_+; \O) \otimes 
 D(\bbR_+; [0,1]) \otimes D( \bbR_+; \bbN)$, whose generic element is given by  $  (\s_t, V_t, N_t)_{t \geq 0}$.} Next, for $(a,\s) \in [0,1]\times \O$ {we}  introduce the measurable functions 
$h^{(\e)}_y(a,\s):=\mathds{1}\left( a \in I_\e( y,\s) \right)$ for $y\in\bbZ^d$, and 
$h^{(\e)}(a,\s):= \sum _{y \in \bbZ} y h^{(\e)}_y(a,\s)$.  

\smallskip

Consider  the filtration $\cF_t $, $t \geq 0$, on $\Theta $ defined by
 \begin{equation}\label{seminario}
 \cF_t= \s ( \s_s :{ s\geq0 } ) \vee    \s ( V_s, N_s: s\in [0,t])
 \end{equation}
 Due to \cite[Thm. T25, App.A.2]{Br}, $(\cF_t)_{t \geq 0}$ is  a right continuous filtration. 
 We denote by  
$(\bar{\cF}_t)_{t \geq 0} $  its completion w.r.t.\@ $\cP_\eta$, i.e.\@ $\bar{\cF}_t = \cF_t \vee \cN$ where $\cN$ is the $\s$--algebra of all  events  on $\Theta $ with $ \cP_\eta$--zero measure.  Due to \cite[Thm. T25, App.A.2]{Br} $(\bar{\cF}_t)_{t \geq 0} $  is right continuous, and therefore it is a filtration  satisfying  the so-called usual conditions. 

\smallskip 

Setting $\eta_s:= \t_{X^{(\e)}_s} \s_s$,
the construction presented in Section \ref{san_valentino} implies that
\begin{equation}\label{sonno}
X^{(\e)}_t= \int _0^t h^{(\e)}( V_s, \eta_{s^-} )dN_s, \qquad \forall t \geq 0\,, \qquad \cP_\eta\text{--a.s.}.
\end{equation}

{We claim that 
\begin{equation}\label{fullMart}
\hat{M}_t:=  X^{(\e)}_t- \int_0^t j^{(\e)} ( \eta_s) ds \,, \qquad t \geq 0\,,
\end{equation}
is a  vector--valued  martingale w.r.t to the filtered probability space $( \Theta, (\bar{\cF}_t)_{t \geq 0}, \cP_\eta)$. Indeed, this follows from \cite[Theorem 9.12]{Kl}  since $F_n(t),m_n  $ there are simply $1-e^{-\l t}$, and $j^{(\e)}$ (recall that  $\sum _{y \in \bbZ^d}\int_0^t |y| r_\e(y, \cdot ) $ is uniformly bounded by our assumptions, and this allows to check the hypothesis of  \cite[Theorem 9.12]{Kl}).}

Note that,  due to   Assumption \ref{ass:rates}, $j^{(\e)}(\eta)= \sum y r_\e(y, \eta)$ is a well defined  bounded function.

\begin{claim}\label{claim_zero}  $\cP_\eta$--a.s.\@ it holds $\lim _{t\to \infty} \hat{M}_t/t=0$, {for $\mu_\e$--a.e. $\eta$.}
\end{claim}

Before proving the above claim, we conclude the proof of Theorem \ref{late}--(iii).
Recall that the trajectory $(\eta_t:= \t_{F_{\e}(t) }\s_t)_{t \geq 0}$ sampled according to $\cP_\eta$ has law $\bbP_\eta^{(\e)}$.
We know that $\bbP_\mu^{(\e)}$--a.s.\@ $t^{-1} \int_0^t d s \,j^{(\e)}(\eta_s)$ converges to $\mu_\e (j^{(\e)})$ (cf.\@ Theorem \ref{lacho_drom} and Theorem \ref{late}--(i)). 
Hence for $\mu_\e$--a.e.\@ $\eta$ we have that  {$t^{-1} \int_0^t d s \,j^{(\e)}(\eta_s)\to \mu_\e(j^{(\e)} )$} w.r.t $\cP_\eta$. This limit together with the above claim and with 
 \eqref{fullMart} allows to conclude that  $X^{(\e)}_t/t \to \mu_\e(j^{(\e)})$   $ P_{\eta,0}^{(\e)}$--a.s..

\begin{proof}[Proof of Claim \ref{claim_zero}] 
 We fix $i \in \{1,\dots, d\}$ and  let $\hat M_t^{(i)}$ be the $i$--th coordinate of $\hat M_t$.  We point out that
 $\hat M_t $ (and therefore  also
  $\hat M_t^{(i)}$) is a square integrable martingale.
 This follows   from \eqref{fullMart}:  since we have assumed that \eqref{nthMoment} holds with $n=2$, it is simple to check that $\cE_\eta( |X^{(\e)}_t|^2) <+\infty$ for any $t\geq 0$, while  $  j^{(\e)}(\cdot)$ is uniformly bounded.

Due to  \eqref{nthMoment}  with $n=2$,  the functions 
$m^{(\e)}(\eta):= \sum _y y_i r_\e (y, \eta) $ and $v^{(\e)}(\eta):= \sum _y y_i^2 r_\e (y, \eta)$ are uniformly bounded. Then, by {\cite[Theorem 9.14]{Kl} (note that $F_n(s)$ and $A(t)$ in \cite[Theorems 9.12, 9.13]{Kl} are given by $1-e^{-\l s}$ and  $\int_0^t j^{(\e)} (\eta_s)ds$, respectively),} we have that 
 \[  \langle \hat M^{(i)}\rangle_t= \int_0^t  v^{(\e)} ( \eta_s) ds\,.\]
 
%
%
%
%
%
%
%
%
   As the dynamic environment is an  ergodic process, $\cP_\eta$--a.s.\@ it holds $  \langle\hat M^{(i)}\rangle_t /t \to \mu_\e (v^{(\e)} )\in (0,+\infty) $ (use  \eqref{nthMoment}   with $n=2$). At this point, by applying the LLN for square integrable martingales (cf.\@ \cite[Thm.1]{Le}), we conclude that $\hat{M}_t^{(i) } /t \to 0$  $\cP_\eta$--a.s. for $\mu_\e$--a.e. $\eta$. 
\end{proof}


\section{Proof of Theorem \ref{porte-bonheur}, Theorem \ref{canicule} and Lemma \ref{commutare}  }\label{tao}

\subsection{Proof of Theorem \ref{porte-bonheur} }
 Define 
 $\cG_n$  as the $\s$--algebra on $\O$ generated by $(\eta_x\,:\, |x| _\infty\leq n)$. Then the smallest $\s$--algebra containing each  $\cG_n$  is the standard Borel $\s$--algebra $\cG$  on $\O$.

Let $ h \in L^1(\mu)$. By L\'evy's upward theorem (cf.\@ \cite{W}[page 134]), as $n \to \infty$ it holds
\begin{equation}\label{willy}
\mu( h | \cG_n)\to  \mu(h|\cG)= h \text{ in } L^1(\mu).
\end{equation}
We then claim that, for any  local function $f$, it holds
\begin{equation}\label{limoncello}
 \lim _{|x| \to \infty} \mu ( h \t_x f ) = \mu(h) \mu(f) \,.
\end{equation}
To prove our claim, we need to show that 
$\lim _{|x| \to \infty} \mu ( h \t_x [f -\mu(f) ] ) = 0$. Equivalently we need to show that 
 $\lim _{|x| \to \infty} \mu ( h \t_x f ) = 0$ if $\mu(f)=0$. 
In particular, we can restrict the proof to the case $f$ local  function with  zero mean. 
We can write
\begin{equation}\label{eq1}
\mu ( h \t_x f )= \mu ( [ h-\mu(h|\cG_n) ] \t_x f )+ \mu \left( \mu(h|\cG_n) \t_x f\right)
\,.
\end{equation} 
By \eqref{willy}, given $\e>0$ we can find $n$ large enough that  $ \mu (|h-\mu(h|\cG_n)|) \leq \e$. Fix such $n$. Note that since $\mu$ is translation invariant we have 
$\mu(\t_x f ) = \mu(f)=0$, which implies that  $ \mu \left( \mu(h|\cG_n) \t_x f\right)=  {\rm Cov} _\mu ( \mu(h|\cG_n) , \t_x f) $. 
 Then \eqref{eq1} gives
 \begin{equation}
\begin{split}
|\mu (h \t_x f) | & \leq \| f\|_\infty \mu (|h-\mu(h|\cG_n)|) + \left | {\rm Cov} _\mu ( \mu(h|\cG_n) , \t_x f) \right|\\
& \leq \|f\|_\infty \e+ \left | {\rm Cov} _\mu ( \mu(h|\cG_n) , \t_x f) \right|\,.
\end{split}
\end{equation}
At this point, using \eqref{zietto}, we conclude that 
\[ \limsup _{|x| \to \infty} |\mu (h \t_x f) |   \leq \|f\|_\infty \e\,.
\]
By the arbitrariness of $\e$ we get our claim.

Having the above claim, Theorem \ref{porte-bonheur} becomes immediate. Indeed, it is enough to \verde{take} $h:= d\mu_\e /d\mu$. Then \eqref{limoncello} becomes equivalent to \eqref{belgio}.

\subsection{Proof of Lemma \ref{commutare}}
Recall the coupling introduced in Section~\ref{san_valentino}. Since the rates $r(\cdot,\eta)$ do not depend on $\eta$, for $\eta\in\O$, under $\cP_\eta$, $X_\cdot$ and $\s_\cdot$ are independent. Therefore, fixed  a local function $f$, $t\in\bbR_+$, $\eta\in\O$, we have 
\begin{equation*}
\begin{split}
\bbE_\eta[f(\t_x\eta_t)]&= \sum_{y\in\bbZ^d}\cE_\eta\bigl[\mathbf{1}_{X_t=y}f(\tau_{x+y}\s_t)\bigr]
=\sum_{y\in\bbZ^d}\cP_\eta(X_t=y)\bbE_\eta^{\rm env}[f(\tau_{x+y}\s_t)] \\
&=\sum_{y\in\bbZ^d}\cP_\eta(X_t=y)\bbE_{\t_x\eta}^{\rm env}[f(\tau_{y}\s_t)] 
=\sum_{y\in\bbZ^d}\cE_{\t_x\eta}\bigl[\mathbf{1}_{X_t=y}f(\tau_{y}\s_t)\bigr]
 =\bbE_{\t_x\eta}[f(\eta_t)]\,.
\end{split}
\end{equation*}

\subsection{Proof of Theorem \ref{canicule}}
 Recall the definition of the functions in  \eqref{Sn} and  that $\bbE_\eta$ denotes the expectation for the   \emph{environment seen by the unperturbed walker} starting from $\eta$, that is, the Markov process with generator \eqref{EnvFromURW}.  To shorten the notation,  we set $f_x := f \circ \t_x$ for all  $x \in \bbZ^d$. 

The proof of Theorem \ref{canicule}  is based on the following technical result:
\begin{Lemma}\label{nasty}
If $
S(t)(g\circ \tau_x)=(S(t)g)\circ\tau_x $
holds for any $g$ local, $t\geq 0$, $x\in\bbZ^d$, then for all $n\geq 0$,  $f: \O \to \bbR$ local function and  $\eta\in\O$, 
it holds
\begin{equation}\label{pingu}
\begin{split}
\verde{\hat{L}_\e S_\e^{(n)}(t)f(\eta)}= &  \int_0^tdt_1\int_0^{t_1}dt_2\ldots\int_0^{t_{n-1}}dt_n\sum_{{z\in B(R)^{n+1},}}\sum_{\d\in\lbrace 0,1\rbrace^{n+1}}(-1)^{|\d|}\\
& \times \bbE_\eta\left[\left(\prod_{i=1}^{n+1}\hat{r}_\e\left(z_i,\tau_{(\d\cdot z)_{[i-1]}}\eta_{t-t_{i-1}}\right)\right)f_{(\d\cdot z)_{[n+1]}}\left(\eta_{t}\right)\right]\,,\end{split}
\end{equation}
where $|\d|=\sum_{i=1}^{n+1} (1-\d_i)$, $(\d\cdot z)_{[i]}=\d_1z_1+\ldots+\d_i z_i$, $(\d\cdot z)_{[0]}=0$ and $t_0=t$.
\end{Lemma}
Formula \eqref{pingu} has to be thought with  no time integration in the  degenerate case $n=0$.

\begin{proof} \verde{For simplicity of notation, as in \eqref{gigino}, we set $g_{n+1}(t):= S_\e^{(n)}(t)f(\eta), n\geq 0$.} 
The proof is done by induction on $n$ and relies on the following two identities (based on the definition of $\hat{L}_\e$, $g_{n+1}(t)$ and Assumptions (i)\footnote{Note that Assumption (i) can be restated as   $\bbE_\eta \bigl(f (\t_x\eta_t) \bigr) = \bbE_{\t_x \eta} \bigl( f( \eta_t) \bigr) $ for any local function $f$ and $x \in \bbZ^d$}  and (iii) in Theorem \ref{canicule}):
\begin{align}
& \hat{L}_\e g_1(t )(\eta)= \sum_{z\in B(R)}\hat{r}_\e(z,\eta)\bbE_\eta\left[f(\tau_z\eta_t)-f(\eta_t)\right] \,,\label{game1}\\
& \hat{L}_\e g_{n+1}(t )(\eta)=\sum_{z_0\in B(R)}\hat{r}_\e(z_0,\eta)\int_0^t dt_1\,\bbE_\eta\left[\hat{L}_\e g_{n}(t_1 )\left(\tau_{z_0}\eta_{t-t_1}\right)-\hat{L}_\e g_{n}(t_1)\left(\eta_{t-t_1}\right)\right] \,. \label{game2}
\end{align}
Trivially \eqref{game1} corresponds to \eqref{pingu} for $n=0$.
Let us now assume \eqref{pingu} holds for $n-1$, where $n \geq 1$, and deduce that it holds for $n$.  
By the induction hypothesis
 we get
 \begin{eqnarray}
\bbE_\eta\left[\hat{L}_\e g_{n}(t_1  )\left(\tau_{z_0}\eta_{t-t_1}\right)\right]=\int_0^{t_1 }dt_2\int_0^{t_2}dt_3\ldots\int_0^{t_{n-1}}dt_n\sum_{{z\in B( R)^n}}\sum_{\d\in\lbrace 0,1\rbrace^{n}}(-1)^{|\d|}
\nonumber\\
\quad\displaystyle\times\,\bbE_\eta\Big[\bbE_{\tau_{z_{0}} \eta_{t-t_1}}\Big[\Big(\prod_{i=1}^{n}\hat{r}_\e\bigl(z_i,\tau_{(\d\cdot z)_{[i-1]}}\eta_{t_1-t_{i}}\bigr)\Big)f_{ (\d\cdot z)_{[n]}}\bigl(\eta_{t_1}\bigr)\Big]\Big].\label{alexey}
\end{eqnarray}
By  Assumption (i) in Theorem \ref{canicule} and the Markov property applied at time $t-t_1$, we can rewrite the expectation in the r.h.s. of \eqref{alexey} as 
 \begin{equation*}
 \begin{split}
&\bbE_\eta\Big[ \bbE_{ \eta_{t-t_1}}\Big[ \Big(\prod_{i=1}^{n}\hat{r}_\e\bigl(z_i,\tau_{z_0+(\d\cdot z)_{[i-1]}}\eta_{t_1-t_{i}}\bigr)\Big)f_{ z_0+(\d\cdot z)_{[n]}}\bigl(\eta_{t_1}\bigr) \Big]\Big]\\
& =\bbE_\eta\Big[  \Big(\prod_{i=1}^{n}\hat{r}_\e\bigl(z_i,\tau_{z_0+(\d\cdot z)_{[i-1]}}\eta_{t -t_{i}}\bigr)\Big)f_{ z_0+(\d\cdot z)_{[n]}}\bigl(\eta_{t}\bigr) \Big]\Big]
 \end{split}
 \end{equation*}
 If we set $(z_1', z_2'\dots, z'_{n+1}):= (z_0,z_1, \dots, z_n)$ and  $(\d_1', \d_2', \dots, \d_{n+1}'):= (1, \d_1, \d_2 ,\dots, \d_n)$, recalling the convention $t_0:=t$ we can write 
 %
  \begin{equation*}
  \begin{split}
  &  \hat{r}_\e(z_0,\eta)\Big(\prod_{i=1}^{n}\hat{r}_\e\bigl(z_i,\tau_{z_0+(\d\cdot z)_{[i-1]}}\eta_{t-t_{i}}\bigr)\Big) f_{ z_0+(\d\cdot z)_{[n]}}\bigl(\eta_{t}\bigr)    \\ =
 &\prod_{i=1}^{n+1}\hat{r}_\e\left(z_i',\tau_{(\d'\cdot z')_{[i-1]}}\eta_{t-t_{i-1}}\right)f_{(\d'\cdot z')_{[n+1]}}\bigl(\eta_{t}\bigr)  \,.
 \end{split}
  \end{equation*}
Coming back to \eqref{alexey} the above observations imply that 
\[
\sum_{z_0\in B(R)}\hat{r}_\e(z_0,\eta)\int_0^t dt_1\,\bbE_\eta\left[\hat{L}_\e g_{n}(t_1 )\left(\tau_{z_0}\eta_{t-t_1}\right)\right]
\]
is given by the r.h.s. of \eqref{pingu} where the sum among $\d$ is restricted to $\d\in B(R)^{n+1}$ with $\d_1=1$. By the same arguments, we get that \[
\sum_{z_0\in B(R)}\hat{r}_\e(z_0,\eta)\int_0^t dt_1\,\bbE_\eta\left[\hat{L}_\e g_{n}(t_1 )\left(\eta_{t-t_1}\right)\right]
\]
is given by the r.h.s. of \eqref{pingu} where the sum among $\d$ is restricted to $\d\in B(R)^{n+1}$ with $\d_1=0$.
To get the thesis is now enough to invoke \eqref{game2}.
\end{proof}

We can now get our estimates for the convergence. To simplify the notation, in what follows we write $|\cdot |$ for the uniform norm $|\cdot |_\infty$.  

\begin{claim}\label{ventilo}
It is enough to show that for all $n,t,x,$
\begin{equation}\label{tropchaudx}\verde{\Bigl|\mu\Bigl(\hat{L}_\e S_\e^{(n)}(t)f_x\Big)\Bigr|}\leq C_0\frac{(Ct)^n}{n!}e^{-\theta_2|x|+\theta_3 t}
\end{equation}
for some $\theta_2>0$, $\theta_3,C,C_0\in\bbR_+$ not depending on $n,t,x$ (possibly depending on $\epsilon, \g$).
\end{claim}

\begin{proof}[Proof of Claim~\ref{ventilo}]
We need to show that 
\begin{equation}\label{crevaison}
\Bigl|\sum_{n\geq 0}\int_0^\infty dt\,\mu\Bigl(\verde{\hat{L}_\e S_\e^{(n)}(t)f_x}\Big)\Bigr|\leq C(f)e^{-\theta'|x|},
\end{equation}
since the l.h.s. equals $\mu_\e (f)-\mu(f)$ {by \eqref{muinfinity_bisXXX}}. Notice that, in addition to \eqref{tropchaudx}, by \eqref{caffettino} we can bound  
\begin{equation}\label{tropchaudnt}
\Big|\mu\Big(\verde{\hat{L}_\e S_\e^{(n)}(t)f_x}\Big)\Big|\leq \e e^{-\g t}\frac{(\e t)^{n}}{n!}\|f\|\,.
\end{equation}


We can estimate
\begin{eqnarray}
\Bigl|\sum_{n\geq 0}\int_0^\infty dt\,\mu\Bigl(\hat{L}_\e S_\e^{(n)}(t)f_x\Big)\Bigr|\leq  C_1 \sum_{n\geq 0}\int_0^\infty dt\,\Bigl(\frac{(Ct)^n}{n!}e^{-\theta_2|x|+\theta_3 t}\Bigr)\wedge\Bigl(e^{-\g t}\frac{(\e t)^{n}}{n!}\Bigr),
\end{eqnarray}
where $C_1=C_0\vee (\e\|f\|)$. Let $\a\in(0,\theta_2/(C+\theta_3))$. The sum in the right-hand side above can be estimated by
\begin{equation*}
\begin{split}
\sum_{ {n\geq 0} }\int_0^{\a|x|}dt \,\frac{(Ct)^n}{n!}e^{-\theta_2|x|+\theta_3t}+\sum_{ {n\geq 0} }\int_{\a|x|}^\infty dt \,e^{-\g t}\frac{(\e t)^{n}}{n!}\\
{= \sum_{{n\geq 0 }}\frac{C^n(\a|x|)^{n+1}}{(n+1)!}e^{-\theta_2|x|+\theta_3\a|x|}+ \int_{\a|x|}^\infty dt\, e^{-(\g-\e) t}}\\
\leq  {\frac{1}{C}} e^{-(\theta_2-\a C-\a\theta_3)|x|}+{ \frac{1}{\g-\e} e^{-(\g-\e)\a|x|} }
\,.
\end{split}
\end{equation*}
Therefore we can indeed choose $\theta'>0 $ as in \eqref{crevaison}.
\end{proof}

We now move to prove \eqref{tropchaudx}. 
{We claim that}, in view of Lemma~\ref{nasty}, it is enough to show that there exist $\theta_2>0, C,\theta_3\in\bbR_+$ such that $\forall n\geq 1$, $\forall t=t_0>t_1 > \cdots >   t_n$, $\forall z\in B(R)^n$, $\forall z_{n+1}\in B(R)$, {$\forall \d'\in\lbrace 0,1\rbrace^n$}, we have:
\begin{eqnarray}\label{gradpingu}
\begin{split}
\bbE_\mu\Big[\Big(\prod_{i=1}^{n+1}r_\e\bigl(z_i,\tau_{(\d'\cdot z)_{[i-1]}}\eta_{t-t_{i-1}}\bigr)\Big)\Bigl(f_{x+({\d'}\cdot z)_{[n]}+z_{n+1}}\bigl(\eta_{t}\bigr)-f_{x+({\d'}\cdot z)_{[n]}}\bigl(\eta_{t}\bigr)\Bigr)\Big]\\
\leq C^ne^{-\theta_2|x|+\theta_3 t}\,.
\end{split}
\end{eqnarray}

To see why this is enough, consider the sum indexed by the $(n+1)$-uples $\delta\in\lbrace 0,1\rbrace^{n+1}$ appearing in \eqref{pingu} and reindex it by gathering together the two terms sharing the same first $n$ coordinates. {To use \eqref{gradpingu}, we set $\d'$ to be the corresponding $n$-coordinate vector.} The integrals and sums contribute at most a factor $\frac{t^n}{n!}(2(2R+1)^d)^{n+1}$ and therefore we get \eqref{tropchaudx} from \eqref{gradpingu} {(changing the value of $C$)}.

In order to prove \eqref{gradpingu}, {we abbreviate} 
\[ \Pi =\prod_{i=1}^{n+1}r_\e\bigl(z_i,\tau_{(\d\cdot z)_{[i-1]}}\eta_{t-t_{i-1}}\bigr),\qquad \Delta f_x=f_{x+(\d\cdot z)_{[n]}+z_{n+1}}\bigl(\eta_{t}\bigr)-f_{x+(\d\cdot z)_{[n]}}\bigl(\eta_{t}\bigr)\,.\]
To conclude we need to show that 
$\bbE_\mu\left[\Pi\Delta f_x\right]\leq C^ne^{-\theta_2|x|+\theta_3 t}$. 

Let $\beta(x)=|x|/5$. Given $\s \in \{0,1 \}^{\bbZ^d}$ we write $\tilde \s $ for the configuration obtained  from $\s$ by periodizing $\s $  restricted to  the box $B( 2\beta(x))$ (for simplicity of notation we assume $\beta(x)$ to be integer). Recall that in Section \ref{san_valentino} we have built the random walk $X_\cdot$ as a function $F \left( \s_\cdot , \cT, \cU\right)$ of the environment trajectory $\s_\cdot$, Poisson times $\cT$ (with parameter $\l$) and uniform random variables $\cU$ (these last objects defined on a probability space with probability measure $\cP_\mu$).  Let $\verde{(\tilde X_s)_{s\geq 0}} := 
 F \left( (\tilde{\s}_s)_{s \in \bbR_+}, \cT, \cU\right)$ and let $N$ be the cardinality of 
 $ \cT \cap [0,t] $. By the definition of $F$ and since the jump rates have finite range $R$ and support of size $R$,  we have  for any $u \leq t$ that $|X_u 
| \leq N R$ and $X_u=  F \left( \s_\cdot , \cT, \cU\right) $ depends on $\s_\cdot$ only through  $\left({\s_s }_{| B((N+1) R)}\right)_{ s \in [0,t]}$.   Note that  $N $ is a Poisson variable with parameter $\l t$ and in particular $P( N > k) \leq e^{- k+(e-1)\l t }$. Hence, taking $k= \b(x)/R-1$, for $x$ large enough  it holds $\cP_\mu(\cG)\geq 1-ee^{- \b(x)/R+(e-1)\l t }$
where 
 $ \cG:= \{ X_s = \widetilde{X}_s \in B( \beta(x)) \; \forall s \leq t \}$.

Let us set 
  \begin{align*}
\widetilde{\Pi}&=\prod_{i=1}^{n+1}r_\e\bigl(z_i,\tau_{(\d\cdot z)_{[i-1]}+\widetilde{X}_{t-t_{i-1}}}   \sigma_{t-t_{i-1}}\bigr),\\
\widetilde{\Delta f_x}& =f_{x+(\d\cdot z)_{[n]}+z_{n+1}+\widetilde{X}_t}\bigl(  \sigma_{t}\bigr)-f_{x+(\d\cdot z)_{[n]}+\widetilde{X}_t}\bigl(    \sigma_{t}\bigr)
\end{align*}
and let us introduce  the  event $\cB:= \{ \tilde{X}_{t-t_1},...,\tilde{X}_{t-t_n},\tilde{X}_t \in B(\beta(x) ) \}$.
 Since $ \cG \subset \cB$ it holds $\cP_\mu(\cB)\geq 1-ee^{- \b(x)/R+(e-1)\l t }$.

Since $\Pi$, $ \D f_x$, $\widetilde{\Pi}$, $\widetilde{\Delta f_x}$   are bounded uniformly in $x$ (respectively by $R^{n+1}$, $2\|f\|_\infty$, $R^{n+1}$, $2\|f\|_\infty$),
writing $\mathbf{\widetilde{X}}:=(\tilde{X}_{t-t_1},...,\tilde{X}_{t-t_n},\tilde{X}_t)$ and $y_0=0$, 
 we can estimate
\begin{equation*}
\begin{split}
|\bbE_\mu[\Pi\Delta f_x]| &\leq |\cE_\mu[\Pi\Delta f_x \mathds{1}_{\cG} ]| +C(f){R^{n+1}}e^{- \b(x)/R+(e-1)\l t }\\
&{=} |\cE_\mu[\widetilde{\Pi}\widetilde{\Delta f_x}  \mathds{1}_{\cG} ]| +C(f){R^{n+1}}e^{- \b(x)/R+(e-1)\l t }\\
&\leq  |\cE_\mu[\widetilde{\Pi}\widetilde{\Delta f_x}   \mathds{1}_{\cB} ]| +2C(f){R^{n+1}}e^{- \b(x)/R+(e-1)\l t }\\
&\leq\Bigl|\sum_{\mathbf{y}\in B(\beta(x))^{n+1}}\mathcal{E}_\mu[\overline{\Pi}\,\overline{\Delta f_x}\mathbf{1}_{\mathbf{\widetilde{X}}=\mathbf{y}}]\Bigr|+2C(f){R^{n+1}}e^{- \b(x)/R+(e-1)\l t } \\
&\leq\Bigl|\sum_{\mathbf{y}\in B(\beta(x))^{n+1}}\mathbb{E}_\mu^{\mathrm{env}}[\overline{\Pi}\,\overline{\Delta f_x}\mathcal{P}_\mu({\mathbf{\widetilde{X}}=\mathbf{y}}|(\sigma_s)_{s\leq t})]\Bigr|+2C(f){R^{n+1}}e^{- \b(x)/R+(e-1)\l t }\\
&\leq \Bigl|\sum_{y_{n+1}\in B(\beta(x))}\mathbb{E}_\mu^{\mathrm{env}}\Bigl[\overline{\Delta f_x}\sum_{\mathbf{y}\in B(\b(x))^n}\overline{\Pi}\,\mathcal{P}_\mu({\mathbf{\widetilde{X}}=(\mathbf{y},y_{n+1})}|(\sigma_s)_{s\leq t})\Bigr]\Bigr|\\
&\quad+\,2C(f){R^{n+1}}e^{- \b(x)/R+(e-1)\l t }\,,
\end{split}
\end{equation*}
where
\begin{align*}
 \overline{\Pi}& :=\prod_{i=1}^{n+1}r_\e\bigl(z_i,\tau_{(\d\cdot z)_{[i-1]}+y_{i-1}}\sigma_{t-t_{i-1}}\bigr)\,,\\
 \overline{\Delta f_x}&:=f_{x+(\d\cdot z)_{[n]}+z_{n+1}+y_{n+1}}\bigl( \sigma_{t}\bigr)-f_{x+(\d\cdot z)_{[n]}+y_{n+1}}\bigl(  \sigma_{t}\bigr)\,,
\end{align*}
and $C(f)=2e\|f\|_\infty$.

 By definition, $\widetilde{X}$  depends only on $(\sigma_s)_{s\leq t}$ restricted to $B( 2 \beta(x))$ (and therefore the same holds for $\mathcal{P}_\mu({\mathbf{\widetilde{X}}=\mathbf{y}}|(\sigma_s)_{s\leq t})$), $\overline{\Pi}$ depends only the process restricted to $B((n+1)R+\beta(x))$ and $\overline{\Delta f_x}$ on $B(x,(n+1)R+L+\beta(x)) \subset B(x,(n+1)R+2\beta(x))$ (for $x$ large, where $B(x,r)= x+B(r)$). { We note that 
$B(x,(n+1)R+2\beta(x))$  and $B((n+1)R+2\beta(x))$ have uniform distance $|x|- 4 \beta (x)-2R(n+1)$,}
   so that by finite speed propagation of the environment process, if $|x|- 4 \beta (x)-2R(n+1)\geq \k t$ (where $\k $ is the constant appearing in the definition of finite speed propagation),
\begin{eqnarray}
\begin{split}
& \bbE_\mu[\Pi\Delta f_x]\leq \Bigl|\sum_{y_{n+1}\in B(\beta(x))}\mathbb{E}_\mu^{\mathrm{env}}[\overline{\Delta f_x}]\mathbb{E}_\mu^{\mathrm{env}}\Bigl[\sum_{\mathbf{y}\in B(\b(x))^n}\overline{\Pi}\,\mathcal{P}_\mu({\mathbf{\widetilde{X}}=(\mathbf{y},y_{n+1})}|(\sigma_s)_{s\leq t})\Bigr]\Bigr|\\
&+\,2C(f){R^{n+1}}e^{- \b(x)/R+(e-1)\l t }+{(2\beta(x)+1)}  R^{n+1}C(f)e^{-\theta(|x|- 4 \beta (x)-2R(n+1))}\\
&= 2C(f){R^{n+1}}e^{- \b(x)/R+(e-1)\l t }+{(2\beta(x)+1)} R^{n+1}C(f)e^{-\theta(|x|- 4 \beta (x)-2R(n+1))}\\
&\leq 2C(f){R^{n+1}}e^{- |x|/5R+(e-1)\l t }+{ c R}e^{2R\theta}C(f)(Re^{2R\theta})^ne^{-\theta|x|/6}\,,\label{greece}
\end{split}
\end{eqnarray}
where the equality follows from $\bbE_\mu^{\mathrm{env}}[\overline{\Delta f_x}]=0$  (since $\mu$ is translation invariant),  we have estimated $\|\sum_{\mathbf{y}\in B(\b(x))^n}\overline{\Pi}\,\mathcal{P}_\mu({\mathbf{\widetilde{X}}=(\mathbf{y},y_{n+1})}|(\sigma_s)_{s\leq t})\|_\infty$ by $R^{n}$ {and the constant $c$ is such that  $( 2u/5+1)e^{- \theta u/5}  \leq e^{ - \theta u/6}$.}

On the other hand, if $|x|- 4 \beta (x)-2R(n+1)\leq \k t$, we can estimate
\begin{equation}
\begin{split}
\bbE_\mu[\Pi\Delta f_x]&\leq 2\|f\|_\infty R^{n+1}=  2\|f\|_\infty R^{2(n+1)}R^{-(n+1)} \\
& \leq 2\|f\|_\infty R^{2(n+1)}R^{-\frac{|x|-4\b(x)}{2R}+\k t} \leq  2\|f\|_\infty R^2 (R^2)^ne^{-|x|\frac{\ln R}{10R}+t \k\ln R}\,.\label{bailout}
\end{split}
\end{equation}
In both cases \eqref{greece} and \eqref{bailout}, we find an estimate of the form \eqref{gradpingu}, which concludes the proof.

\section{{Proof of Theorem \ref{QCLT}--(i)} }\label{RWproof_pera}

Recall the notation of  Sections  \ref{san_valentino} and \ref{lavatrice}. 
We introduce the probability measure $ \cP_{\mu_\e}:= \int \mu_\e (d \eta) \cP_{ \eta}$ on the space $\Theta$. We write $Q$ for the image of $\cP_{\mu_\e}$ induced by  
 the map $\Theta \ni ( \s_t, V_t, N_t) _{t \geq 0} \mapsto     (\t_{X_t^{(\e)} } \s_t ,V_t, N_t)_{t \geq 0 } \in \Theta$.  Note that the projection of $Q$ along the first coordinate  is simply $\bbP^{(\e)}_{\mu_\e}$. To stress this property, we write $(\eta_t, V_t,N_t) _{{t \geq 0}}$ for a generic element of probability space $(\Theta,Q)$, since usually we set $\eta_t:=  \t_{X_t^{(\e)} } \s_t$.
 
 Given $t \geq 0$, we define $\cH_t= \s (   \eta_s , V_s, N_s \,:\, s \in [0,t])$  as $\s$--algebra on $\Theta$. 
Then we write $\bar{\cH}_t$ for the augmented filtration w.r.t.\@ $Q$ following \cite[Def.7.2, Sec.2.7]{KS}. Since  $( \eta_s , V_s, N_s )_{t \geq 0}$ is a strong Markov process,  by  \cite[Prop.7.7, App.A]{KS}  the filtration  $( \bar{\cH}_t)_{t\geq 0}$ on $(\Theta, Q)$ satisfies the usual conditions.
  
 By   the martingale representation in \eqref{fullMart} and since $v(\e)= \mu_\e ( j^{(\e)})$ (cf. Theorem \ref{late}--(iii)),  the position of the walker  centered with its asymptotic velocity  can be written as 
\begin{equation}\label{centrato}
X^{(\e)}_t-v(\e)t=\hat{M}_t+\int_0^t d s \,[j^{(\e)}(\eta_s)-\mu_{\e}(j^{(\e)})]=:\hat{M}_t+A_t(f)\,,
\end{equation}
where $(\hat{M}_t ) _{t \geq 0}$ is a martingale w.r.t.\@ $( \Theta, ( \bar{\cH}_t)_{t\geq 0} , Q )$ and 
 $A_t(f)$ is the additive functional introduced in  Section \ref{limiti_perturbato}  associated with the function $f(\eta):=j^{(\e)}(\eta)-\mu_\e(j^{(\e)})$.
Note that,  by Assumption~\ref{ass:rates}, 
 the vector--valued function $f$ is a bounded continuous function on $\O$ with $\mu_\e(f)=0$. In particular, following the proof of Proposition~\ref{gelem_gelem},   for each coordinate $i=1,\ldots,d$, we can find  $g^{i}\in\mathcal{D}(\mathcal{L}_\e)$ such that 
 $ f^i= - \cL_\e g^{i}$, thus leading to the decomposition 
\begin{equation}
\int_0^t  f^i (\eta_s) ds =\tilde{M}^{(i)}_t+R^{i}_t,
\end{equation}
where $\tilde{M}^{(i)}_t=g^{i}(\eta_t)-g^i (\eta_0)-\int_0^t d s\, \mathcal{L}_\e g^{i}(\eta_s)$ is a martingale and $R^{i}_t=-   g^{i}(\eta_t)+ g^i (\eta_0)$ satisfies the conclusion of Lemma~\ref{monnezza}. As a consequence, we have that
\begin{equation}
X^{(\e)}_t-v(\e)t=\sum_{i=1}^d M_t^{(i)}e_i+R_t, \qquad M_t := \hat{M}_t+ \tilde{M}_t \,.
\end{equation}
Due to Lemma \ref{monnezza}, to get the thesis we only need to 
  apply \cite[Thm. 2.29, Chp. 2]{KLO}, which will show the invariance principle for the martingale term. Due to  the definition of $\hat M_t, \tilde M_t$ (see also \eqref{sonno} and \eqref{fullMart}), the  martingale clearly has stationary increments on $(\Theta, Q)$ and is square integrable (see the discussion after \eqref{fullMart}). It remains to show that for any $i,j=1,\ldots,d$, $\langle M^{(i)}, M^{(j)}\rangle_{n}/n$ converge a.s.\@ and in $L^1$ to $D_\e(i,j)=E [ M_1^{(i)}M_1^{(j)}]$, denoting by $E$ the expectation w.r.t.\@ $Q$. By the parallelogram identity, it is enough to show that $\langle M^{(i)}\pm M^{(j)}\rangle_{n}/n$ converge a.s.\@ and in $L^1$ to $\bbE^{(\e)}_{\mu_\e} [( M_1^{(i)}\pm M_1^{(j)})^2]$.   
 
   To this aim we set $\cM_t:= M^{(i)}_t\pm M^{(j)}_t$ and observe that  
   \begin{equation}\label{pazienza} 
   \cM_{t+s}=\cM_t+ \cM_s \circ \theta _s \qquad \forall t,s \geq 0\,,
   \end{equation} where $\theta_s$ is the time-translation on $\Theta$ at time $s$. Due to the martingale property, we have 
   \begin{equation}
   E\bigl[\cM_{t+s}^2- \cM_t^2 | \bar{\cH}_t \bigr]=  E\bigl[ (\cM_{t+s}- \cM_t)^2 | \bar{\cH}_t \bigr]= E\bigl[ \cM_{s}^2 | \bar{\cH}_0 \bigr]\circ \theta_t\,.
   \end{equation} 
    By the definition of $\cM_t$ it follows easily that $E\bigl[ \cM_{s}^2 | \bar{\cH}_0 \bigr]$ is $\s(\eta_0)$--measurable. To simplify the notation, we write $F_s( \eta_0)$ for  $E\bigl[ \cM_{s}^2 | \bar{\cH}_0 \bigr]$, thus allowing to write $E\bigl[\cM_{t+s}^2- M_t^2 | \bar{\cH}_t \bigr]=F_s(\eta_t)$. On the other hand, by     \cite[Thm. (31.2), Chp. VI.6.31]{RW2}, $\langle \cM \rangle_t$ is the limit of $ \Sigma^t_S$ in the weak ($L^\infty)$ topology of $L^1( \Theta,Q)$ as the mesh of the partition $S$ goes to zero, where $\Sigma^t_S:= \sum_{i=1}^n  E[ \cM^2_{t_i}-   \cM^2_{t_i-1}| \bar{\cH}_{t_{i-1} }] $ for a partition   $ S= \{ t_0 =0< t_1< \cdots < t_n=t\}$.\footnote{\cite[Thm. (31.2), Chp. VI.6.31]{RW2} is stated for  the compensator of submartingale of class (D), on the other hand  for any fixed $T>0$ the process $\cM^2_{t \wedge T}$ is uniformly integrable and therefore it is a  submartingales of class (D) (combine \cite[Thm. 7.32]{Kl}  and \cite[Lemma (29.6), Chp. VI.6.29]{RW2}).}

     By what we just proved, we can write $\Sigma^t_S=   \sum_{i=1}^n F_{t_i-t_{i-1} }(\eta_{t_{i-1} } )$. As a byproduct, we deduce that, given   $t,s \geq 0$, it holds                
  \begin{equation}\label{fucile}
  \langle \cM \rangle _{t+s} =\langle \cM \rangle _{t} +  \langle \cM \rangle _{s}\circ \theta _t  
\qquad Q\text{--a.s.}  
  \end{equation}  
  Moreover, we get that $\langle \cM \rangle _t$ depends only on $(\eta_s)_{s \leq t}$,  more precisely that $ \langle \cM \rangle _t = E[  \langle \cM \rangle _t | \cG_t]$ where $\cG_t$ is the $\s$--algebra on $\Theta$ generated by $\{ \eta_s \,:\, s \in [0,t]\}$. Indeed, by definition of weak limit and since $\Sigma^t_S$ is $\cG_t$--measurable,  we have
  \begin{equation}
  E\bigl[ \langle \cM \rangle _t \xi \bigr] = \lim_S E\bigl[ \Sigma^t_S \xi \bigr]= \lim _S E\bigl[ \Sigma^t_S E[\xi\,|\, \cG_t] \bigr]=   E\bigl[ \langle \cM \rangle _t E[ \xi\,|\, \cG_t] \bigr]\,, 
  \end{equation}
  for any   bounded random variable $\xi$  on $\Theta$.  
  Since $ E\bigl[ \langle \cM \rangle _t E[ \xi\,|\, \cG_t] \bigr]=  E\bigl[ E\bigl[ \langle \cM \rangle _t \,|\, \cG_t]\xi  \bigr]$, we conclude that 
   $ E\bigl[ \langle \cM \rangle _t \xi \bigr] =     E\bigl[ E\bigl[ \langle \cM \rangle _t \,|\, \cG_t]\xi  \bigr]$ for any $\xi$ as above, thus implying that $ \langle \cM \rangle _t = E[  \langle \cM \rangle _t | \cG_t]$.

 At this point, one can deduce that $\langle \cM \rangle _n/n$ converges 
 a.s.\@ and in $L^1$ to $  E( \cM_1^2)= \bbE^{(\e)}_{\mu_\e} [( M_1(i)\pm M_1(j))^2]$ by the same arguments used in the proof of    
 Lemma~\ref{lufthansa}.

\begin{remark} Note that for $1\leq i\neq j\leq d $, the martingales $M_t^{(i)}$ and $M_t^{(j)}$ have common jumps if the rate $r_\e(y, \cdot)$ is positive for some  $y\in\bbZ^d$ of the form $y=\sum_{i=1}^d c_i e_i$ with $c_i\neq 0 \neq c_j$. Hence, in general, they are not orthogonal, resulting into a non--diagonal diffusion matrix $D_\e$.   
\end{remark}

\section{{Proof of Theorem \ref{QCLT}--(ii)}}
\label{NonDeg}
It remains to show the non--degeneracy of $D_\e$
under the extra hypotheses {  that  $L_{\rm env}$ and $L_{\rm ew}$ are self--adjoint in $L_2(\mu)$,  and  that  \eqref{nthMoment} holds with $n=4$. 
We refer to the notation introduced in the previous section and in Section \ref{san_valentino}. 
 One consequence of the previous proof is that
$\lim _{t \to \infty} \frac{1}{t}{ Var_{\mu_\e}} \left[ X^{(\e)}_t\cdot e \right]=\lim _{t \to \infty} \frac{1}{t} \cE_{\mu_\e} ( (M_t\cdot e) ^2 ) $, where 
${Var_{\mu_\e}}$ denotes the variance w.r.t  
{$\cP_{\mu_\e}$}.
On the other hand,  it holds  $\cE_{\mu_\e} ( (M_t\cdot e) ^2 )= \cE_{\mu_\e} ( \langle  M\cdot e \rangle_t  )=t \langle e,D_\e e \rangle$ (for the last identity see the conclusion of the previous section). Hence, to prove the  non--degeneracy of $D_\e$ it is enough to show that $\lim _{t \to \infty} \frac{1}{t} {Var_{\mu_\e}}\left[ X^{(\e)}_t\cdot e \right]$  
 is bounded away from zero}.

Along this proof we will heavily use the coupling and the notation introduced in Section \ref{san_valentino}. 
We further define {$\eta^{(\e)}_t:=\t_{X_t^{(\e)}}\s_t$ and $\eta_t:=\t_{X_t }\s_t$}. 

Denote by {$(\cH_t)_{t\geq0}$ the filtration on $\Theta$  with $\cH_t= \s (   \eta_s , V_s, N_s \,:\, s \in [0,t])$  as in Section \ref{lavatrice}}. 
Fixed  a vector $e\in \bbR^d \setminus \{0\}$ and  an integer $T>0$, define the discrete-time martingale $(M^T_n)_{ 0 \leq n \leq T}$ as {(cf.\@ \eqref{fullMart})} \begin{equation*}
\begin{split}
M_n^T &:= {\cE_{\eta}} \left[  X^{(\e)}_T\cdot e  |
{\cH_{n}}  \right] - {\cE_{\eta}}\left[ X^{(\e)}_T\cdot e \right]\\
&= X^{(\e)}_n  \cdot e  + \int_0^{T-n} S_{\e}(s)j^{(\e)}_{e}(\eta^{(\e)}_n)\,ds - \int_0^T S_{\e}(s)j^{(\e)}_{e}(\eta)\,ds,\end{split}\end{equation*}
with $j^{(\e)}_{e}(\eta)=\sum_{y\in\mathbb{Z}^d}(y\cdot e) r_\e(y,\eta),\eta\in\O$. Since \[
{Var_{\mu_\e}}\left[ X^{(\e)}_T\cdot e \right]  ={\cE_{\mu_\e} } \left(  Var^{(\e)}_{\mu_\e}\left[ X^{(\e)}_T\cdot e  \,|\, {\cH_0}\right]\right) +{  Var_{\mu_\e}} \left ( {\cE_{\mu_\e}} \left[ X^{(\e)}_T\cdot e  \,|\, {\cH_0}\right]\right) \,,
\]
by using the above martingale, the stationarity of the perturbed process under $\mu_\e$, and the semigroup property, we can estimate
\begin{equation}
\begin{split}
\label{VarBound}
& {Var_{\mu_\e}} \left[ X^{(\e)}_T\cdot e \right] \geq 
{\cE_{\mu_\e}} \left[\left( X^{(\e)}_T\cdot e -{\cE_{\eta}}\left[ X^{(\e)}_T\cdot e \right] \right)^2 \right]
= \sum
_{n=1}^T {\cE_{\mu_\e}}\left[
(M^T_n-M^T_{n-1} )^2\right]
\\
&= \sum_{n=1}^T {\cE _{\mu_\e}}
\left[\left( X^{(\e)}_{1}\cdot e + \int_0^{T-n} S_{\e}(t)j^{(\e)}_{e}(\eta^{(\e)}_1)\,dt- \int_0^{T-n+1} S_{\e}(t)j^{(\e)}_{e}(\eta)\,dt
\right)^2\right]\\
&=\sum_{n=1}^T{\cE_{\mu_\e}}\left[\left(A^{(\e)}_1+B^{(\e)}_{T-n}\right)^2\right]\,,
\end{split}
\end{equation}
where 
\begin{equation}\begin{split}\label{AB}
&A^{(\e)}_1:=X^{(\e)}_1\cdot e -\int_0^1S_{\e}(t)j^{(\e)}_e(\eta)dt\\
&B^{(\e)}_{T-n}:=\int_0^{T-n}dt\left(S_{\e}(t)j^{(\e)}_e(\eta^{(\e)}_1)-\bbE^{(\e)}_{\eta}\left[S_{\e}(t)j^{(\e)}_e(\eta^{(\e)}_1)\right]\right)\,.
\end{split}
\end{equation}
Note that in the derivation of the last equality in \eqref{VarBound} we have used that 
\begin{equation*}
S_{\e}(t+1)j^{(\e)}_e(\eta)=\bbE^{(\e)}_{\eta}\left[S_{\e}(t)j^{(\e)}_e(\eta^{(\e)}_1)\right].
\end{equation*}


We want to show that $A^{(\e)}_1$ and $B^{(\e)}_{T-n}$ are ``$\e$--close'' to their \emph{unperturbed} counterparts
 $A_1$ and $B_{T-n}$ defined as 
 \begin{equation}\begin{split}\label{AB0}
&A_1:=X_1\cdot e -\int_0^1S(t)j_{e}(\eta)dt\\
&B_{T-n}:=\int_0^{T-n}dt\left(S(t)j_e(\eta_1)-\bbE_{\eta}\left[S(t)j_e(\eta_1)\right]\right),
\end{split}
\end{equation}
where \begin{equation}
j_e(\eta):=\sum_{y\in\bbZ^d} (y\cdot e)r(y,\eta), \quad \eta\in\O,
\end{equation}
\verde{is the unpertubed analogous of the function $j^{(\e)}$ in Theorem \ref{late}.}
Note that due to Assumption \eqref{nthMoment} with $n=4$ it holds $\|j_e\|_\infty <+\infty$.
Having \eqref{VarBound} the rest of the proof is divided in three main steps:

\begin{claim}\label{Claim1}
There exists  $\d(\e)$ going to zero  as $\b(\e) \to 0$ {(recall \eqref{betacarotene})} such that
\begin{equation}\label{claim1}
{\cE_{\mu_\e}}\left[\left(A^{(\e)}_1+B^{(\e)}_{T-n}\right)^2\right]\geq {\cE_{\mu_\e}}\left[\left(A_1+B_{T-n}\right)^2\right]-\d(\e)
\end{equation}
for any $T,n$.
\end{claim}

\begin{claim}\label{Claim2}
There exists a positive constant $C$ such that 
\begin{equation}\label{AtoA0} 
{\cE_{\mu_\e}}\left[\left(A_1+B_{T-n}\right)^2\right]\geq { \cE_{\mu}}\left[\left(A_1+B_{T-n}\right)^2\right]- C\frac{\e }{\g-\e}\
\end{equation}
for any $ T,n$  and $\e < \g$.
\end{claim}

\begin{claim}\label{Claim3}
It holds
\begin{equation}\label{eq_claim3}
\underset{T\to +\infty}{\lim}\frac{ {\cE_\mu}\left[(X_T\cdot e)^2\right]-\mu\left({\cE_{\eta}}\left[X_T\cdot e\right]^2\right)}{T}>0.\end{equation}
\end{claim}

We postpone the proof of the above claims to Sections \ref{dim_claim1}, \ref{dim_claim2} and
\ref{dim_claim3}, and explain how to conclude.
First we  note that, as in the derivation in \eqref{VarBound}, {for the unperturbed processes} we can write
\begin{equation}\label{A0X0}
\sum_{n=1}^T{\cE_{\mu}}\left[\left(A_1+B_{T-n}\right)^2\right]={\cE_\mu} \left[\left(X_T\cdot e -{\cE_{\eta}}\left[ X_T\cdot e\right]\right)^2\right].
\end{equation}
Therefore, by combining \eqref{VarBound}, Claim \ref{Claim1} and Claim \ref{Claim2}, we get that

\begin{equation}\label{lyon}
{Var_{\mu_\e}}\left[ X^{(\e)}_T\cdot e \right]
\geq {\cE_\mu}\left[(X_T\cdot e)^2\right]-\mu\left({\cE_{\eta}}\left[ X_T\cdot e\right]^2\right)-(C\e/(\g-\e) +\d(\e))T.
\end{equation}

Thus, by using \eqref{lyon} together with Claim \ref{Claim3} and  choosing 
$\b(\e)$ small enough, the non-degeneracy of the diffusion matrix is proven.

Before moving to the proofs of the above three claims we collect some  technical facts  that will be repeatedly used below. 


\begin{Lemma}\label{genova}
There exists a function $F(c,n,t)$, where $c>0$, $n$ is  a positive integer and   $t \geq 0$,  such   that $ \sup _{\eta \in \O}{ \cE_{\eta}} \left( |X^{(\e)}_t|^n \right)\leq  F(c,n,t) $  if $\sum_{y\in\bbZ^d}|y|^n\sup_{\eta\in\O} r_\e(y,\eta) \leq c$.
\end{Lemma}

\begin{proof}
 We consider the extended Markov process on $\O \times \bbZ^d \times \bbR_+$ with Markov generator
\begin{equation}
\begin{split}
\cL_\e f(\eta, x, \ell):&=  L_{\rm env} f(\eta,x,\ell)+ \sum_{y \in \bbZ^d}r_\e(y,\tau_{x}\eta) \big[f(\eta,x+y, \ell+|y|)-f(\eta, x,\ell)\big]\\
& +  \sum_{y \in \bbZ^d}\{  \sup_{\z \in \O} r_\e(y, \z)  -r_\e(y, \tau_x \eta)\}  \big[f(\eta,x, \ell+|y|)-f(\eta, x,\ell)\big]\,.
\end{split}
\end{equation}
Note that $\cL_\e$ acts as $L^{(\e)}_{\rm rwre}$ on functions $f$ depending  on $\eta,x$ only, while on  functions $f=f(\ell)$ it reads $\cL_\e f(\ell)= \sum _{v} \verde{R_\e(v)} \bigl[ f(\ell+v)- f(\ell)\bigr]$,  where $v$ varies in $V:= \{ |y|\,:\, y \in \bbZ^d\}$ and 
\[ \verde{R_\e(v)}:= \sum _{y \in \bbZ^d\,: \, |y|=v} \sup_{\eta \in \O} r_\e(y, \eta)  \,.\] 
Hence, the extended Markov process with generator $\cL_\e$ gives a coupling between the \verde{joint process with generator $L^{(\e)}_{\rm rwre}$} and a jump process $(Z_t)_{t \geq 0}$ on $\bbR_+$ with jump probability rates $ \verde{R_\e(\cdot)}$. Moreover, by construction, $Z_t \geq |X^{(\e)}_t|$ for any time $t$ if $Z_0\geq
|X^{(\e)}_0|$. Starting the extended Markov generator at $(\eta,0,0)$, we conclude that 
${\cE_{\eta}} \bigl( |X^{(\e)}_t|^n \bigr)\leq \bbE( Z_t^n)\,.$ 

It remains to bound $\bbE( Z_t^n)$. To this aim we   define $\verde{\l_\e}:= \sum _{v} \verde{R_\e(v)}\leq c $ and take a sequence  $U_1,U_2, \dots $ of  i.i.d. random variables taking value in $V$ with $\bbP(U_i=v)= \verde{R_\e(v)}/\verde{\l_\e}$.  Our main hypothesis implies that  $\bbE [U_i^n]\leq \sum_{y\in\bbZ^d}|y|^n\sup_{\eta\in\O} r_\e(y,\eta)\leq c$.
Taking an independent Poisson process $(N_t)_{t \geq 0}$ of parameter $\verde{\l_\e}$ and setting $S_n:=U_1+ \cdots+U_n$,  
 the  process $Z_t$
 can   be written as $S_{N_t}$. In particular, we have 
$ \bbE( Z_t^n) =  
\sum _{k=0}^\infty \bbP(N_t=k) \bbE [(U_1+\cdots+U_k)^n]$. 
  By H\"older inequality, it holds 
   $(U_1+\cdots+U_k)^n \leq k^{n-1} (U_1^n+\cdots+U_k^n) $.
   Hence, we conclude that 
   $ \bbE( Z_t^n)  \leq  \bbE[U_i^n] \bbE[ N_t ^{n  }]$ leading to the thesis.
\end{proof}

Since the positivity of $D_\e$ has to be proved for $\e$ small enough, in the rest of this section  we assume $ \e \leq \g/2$ so that the term $ 1/(\g-\e)$ is uniformly bounded. 
\begin{Lemma}\label{spina} The expected values  {$\cE_{\mu}\bigl[A_1^4\bigr]$, $\cE_{\mu_\e} \bigl[A_1^4\bigr]$,   
$\cE_{\mu_\e}^{(\e)}\bigl[ \bigl(A^{(\e)}_1 \bigr)^4\bigr]$} are bounded  from above uniformly in $\e$. 
The expected values
{$\cE_{\mu}\bigl[B_{T-n}^4\bigr]$, $\cE_{\mu_\e}\bigl[B_{T-n}^2\bigr]$, $\cE^{(\e)}_{\mu_\e}\bigl[\bigl(B^{(\e)}_{T-n}\bigr)^4\bigr]$} are bounded from above uniformly in $\e,T,n$.
 \end{Lemma}
 \begin{proof}
 
The term {$\cE_{\mu}\bigl[A_1^4\bigr]$}  is bounded since  $\int_0^1S(t)j_{e}(\eta)dt$ is bounded in uniform norm (as $j_{e}$ is bounded in uniform norm), 
and since { $\cE_{\mu}\bigl[( X_1 \cdot e )^4\bigr]$} is bounded (as application of modified version of  Lemma \ref{genova} with a suitable choice of the rates and due to  our condition \eqref{nthMoment}). Similarly, one gets that ${\cE_{\mu_\e}}\bigl[A_1^4\bigr]$ and ${\cE_{\mu_\e}} \bigl[ \bigl(A^{(\e)}_1 \bigr)^4\bigr]$ are bounded from above uniformly in $\e$.

We now  consider the term  ${\cE_{\mu}}\bigl[B_{T-n}^4\bigr]{= \bbE_\mu\bigl[B_{T-n}^4\bigr]} $. To this aim we first observe that, given $k \geq 1 $ and generic numbers $a_1,a_2, \dots , a_k$, by  Schwarz inequality it holds 
 $\bigl(  \sum _{i=1}^k a_i\bigr)^2 \leq c \sum _{i=1}^k i^2 a_i^2$, where 
 $ c:= \sum _{i=1}^\infty \frac{1}{i^2} $.  By applying twice the above inequality we conclude that $ \bigl(  \sum _{i=1}^k  a_i\bigr)^4 \leq c^3
   \sum _{i=1}^k i^6 a_i^4$.  This implies that 
   \begin{equation}\label{sudafrica}
     \bigl( \, \sum _{i=1}^k  a_i\,\bigr)^4 \leq c^3  (\sup _{1\leq i \leq k} | a_i|)^2 \sum _{i=1}^k i^6 a_i^2\,.
     \end{equation}
  Let us come back to $B_{T-n}$.
Note that in the definition of $B_{T-n}$ we can replace $j_{e}$ by $\bar j := j_{e}-\mu(j_{e})$. 
Since $j_{e}$ is uniformly bounded, we have
     $\sup_{t \geq 0} \| S(t)\bar  j \|_\infty< \infty$. By applying \eqref{sudafrica} then we have 
   \begin{equation}\label{giudizio}
   \begin{split}
  \bbE_{\mu}\Big[B_{T-n}^4\Big]\leq C  \sum _{i=1}^{T-n} i^6 \bbE_\mu
  \Big\{\, \Big[\int_{i-1}^{i}\Big(S(t)\bar j(\eta_1)-\bbE_{\eta}\left[S(t)\bar j(\eta_1)\right] \Big)dt \Big]^2\,
   \Big\}\,.
  \end{split}
  \end{equation}
  Above and  in what follows, $C,C'$  denote  positive universal constants (not depending from $T,n, \e$)  that can change from line to line.
By applying Schwarz inequality we have
\begin{equation}\label{giudizio1}
\text{r.h.s. of \eqref{giudizio} } \leq C \sum _{i=1}^{T-n} i^6 \bbE_\mu
  \Big\{
   \int_{i-1}^{i}   
   \Big(
  S(t)\bar j (\eta_1)
  \Big)^2    dt+ \int_{i-1}^{i} 
  \bbE_{\eta}
   \Big[ \Big(
  S(t)\bar j (\eta_1)
  \Big)^2    \Big] 
   dt
   \Big\}\,.
\end{equation}
By stationarity and by the spectral gap of \verde{$L_{\rm ew}$ in $L^2(\mu)$ (cf.\@ \eqref{gap2})}, we then conclude that 
\begin{equation}\label{giudizio2}
\begin{split}
\text{r.h.s. of \eqref{giudizio1} } & \leq C \sum _{i=1}^{T-n} i^6 \int \mu (d\eta) 
  \Big\{
   \int_{i-1}^{i}   
   \Big(
  S(t)\bar j (\eta)
  \Big)^2    dt \Big\}= C \sum _{i=1}^{T-n} i^6    \int_{i-1}^{i} \|    
  S(t)\bar j  \verde{\|}^2 \\&
  \leq C' \sum _{i=1}^{T-n} i^6    \int_{i-1}^{i}  e ^{-2\g t} dt \leq
  C' \sum _{i=1}^{T-n} i^6       e ^{-2\g (i-1)}<+\infty \,.
    \end{split} 
  \end{equation}
  By combining \eqref{giudizio}, \eqref{giudizio1} and \eqref{giudizio2} we get the thesis, i.e.\@  
  $\bbE_{\mu}\bigl[B_{T-n}^4\bigr]$ is bounded from above uniformly in $T$ and $n$. 
  
  By similar arguments, considering now the perturbed process, one can prove that $\bbE_{\mu_\e}^{(\e)}\bigl[\bigl(B^{(\e)}_{T-n}\bigr)^4\bigr]$ is bounded from above uniformly in $T$, $n$ and $\e$. 
 
  We now consider $\bbE_{\mu_\e}\bigl[B_{T-n}^2\bigr]$. 
\verde{By \eqref{benedicte} and since both  $\mu(f) $ and $ \|f -\mu(f)\|$ are bounded by  $\| f \|$ for any $f \in L^2(\mu)$,  we can estimate
  \[
  \bbE_{\mu_\e}\bigl[B_{T-n}^2 \bigr]= \mu_\e\bigl( \bbE_{\cdot}\bigl[B_{T-n}^2\bigr]\bigr)\leq \bigl(1+ \frac{\e}{\g-\e} \bigr)  \|  \bbE_{\cdot}\bigl[B_{T-n}^2\bigr]\| \,.
  \] 
  By Schwarz inequality, we have  $ \|  \bbE_{\cdot }\bigl[B_{T-n}^2\bigr]\|  \leq \int \mu(d\eta)  \bbE_{\eta }\bigl[B_{T-n}^4\bigr]=\bbE_{\mu}\bigl[(B_{T-n}\bigr)^4\bigr]$.} Hence to conclude we invoke that $\bbE_{\mu}\bigl[(B_{T-n}\bigr)^4\bigr]$ is bounded from above uniformly in $T,n$ as just proven. 
\end{proof}

 
\subsection{Proof of Claim~\ref{Claim1}}\label{dim_claim1}
     
Let us start with a simple computation showing that, to get \eqref{claim1}, it is enough to prove that there exists
 $\d(\e)\to 0$ as $\b(\e) \to 0$ such that
\begin{eqnarray}
\mathcal{E}_{\mu_\e}\Big[\Big(A_1^{(\e)}-A_1\Big)^2\Big]&\leq& \delta^2 (\e) \,,
\label{claimA}\\
\mathcal{E}_{\mu_\e}\Big[\Big(B^{(\e)}_{T-n}-B_{T-n}\Big)^2\Big]&\leq& \delta^2(\e)\,.\label{claimB}
\end{eqnarray}
Below $C$ will denote a positive constant, independent from $n,T, \e$.

Since $a^2-b^2= (a-b)(a+b)$  we can bound
\begin{multline*}
 \left| \left(A^{(\e)}_1+B^{(\e)}_{T-n}\right)^2-\left(A_1+B_{T-n}\right)^2 \right| \\
 \leq
\left[\left|A^{(\e)}_1-A_1\right| + \left|B^{(\e)}_{T-n}-B_{T-n}\right| \right] \cdot 
\left| 
A^{(\e)}_1+B^{(\e)}_{T-n}+ A_1+B_{T-n}\right|\,.
\end{multline*}
Using the  above bound and Schwarz inequality we conclude that 
\begin{multline*}
\left|\mathcal{E}_{\mu_\e}\left[\bigl(A^{(\e)}_1+B^{(\e)}_{T-n}\bigr)^2-\bigl(A_1+B_{T-n}\bigr)^2\right]\right|
\\ \leq  c\,  \mathcal{E}_{\mu_\e}\left[(A^{(\e)}_1-A_1)^2\right]^{1/2}+  c\, \mathcal{E}_{\mu_\e}\left[(B^{(\e)}_{T-n}-B _{T-n})^2\right]^{1/2}\,,
\end{multline*}
where  $c:= \mathcal{E}_{\mu_\e}  \left[ \bigl | 
A^{(\e)}_1+B^{(\e)}_{T-n}+ A_1+B_{T-n}\bigr|^2 \right] ^\frac{1}{2} $. Due to Schwarz inequality and Lemma \ref{spina} we conclude that $c$ is bounded uniformly in $T,n,\e$. In particular, to get Claim \ref{Claim1}  it is enough to have \eqref{claimA} and \eqref{claimB}.

  \medskip

Let us now prove  \eqref{claimA}.  \verde{We set  $c(\e):= \sup _\eta  \sum_y | \hat r_\e(y, \eta)| $ and
\begin{equation}\label{En}
E_n:=\{ \eta^{(\e)}_{t_k}= \eta_{t_k}  \; \forall k <n  \text{ and }  t_n \in [0,1]\},\quad  \text{ for } n \geq 1.
\end{equation}} 
We then observe that, writing $\z= \eta^{(\e)}_{t_{n-1}} $, it holds\footnote{$\D$ denotes the symmetric difference, i.e.\@ $A \D B:= (A \setminus B)\cup (B \setminus A)$}
 \begin{equation}
 \begin{split} 
  \cP_\eta( \eta^{(\e)}_{t_n}  \not  = \eta_{t_n}  |E_n) & = \sum _{ y \in \bbZ^d} \cP_\eta\bigl(  U_n \in   I_\e(y, \z)\Delta  I(y,\z)\bigr)   \\ & \leq \sum_{y \in \bbZ^d} \left( | I_\e(y,\z) + | I(y,\z) | - 2 | I_\e(y,\z) \cap I(y,\z) |\right)\\
  & = \l^{-1}\sum _{ y \in \bbZ^d}  |\hat r_\e (y, \z) |\leq \l^{-1} c(\e)\,.
  \end{split}
  \end{equation}
Hence we can estimate \begin{multline}\label{federico}
\cP_\eta\left( \exists s \in [0,1] \text{ s.t. } \eta^{(\e)}_s\not = \eta_s\right)
= \sum _{n=1}^\infty  \cP_\eta( \eta^{(\e)}_{t_n}\not  = \eta_{t_n}  |E_n) \cP_\eta(E_n)  \leq \frac{ c(\e) }{\l}  \sum _{n=1}^\infty \cP_\eta( E_n) \\
\leq \frac{ c(\e)}{\l}    \sum _{n=1}^\infty \cP_\eta( t_n \in [0,1]) = \frac{c(\e) }{\l} 
\cE_\eta ( | \cT \cap [0,1] |)= c(\e)  \,.
\end{multline}
In particular, $\cP_\eta (X^{(\e)}_1\not = X_1) \leq c(\e)$. By Schwarz inequality and Lemma \ref{genova}, which allows with \eqref{nthMoment} to bound the forth moments of $X^{(\e)}_1, X_1$  uniformly in $\e$ (for $X_1$ one has to slightly change the notation in the lemma), we get
 \begin{equation}\label{martina}
 \cE_{\mu_\e} \Big[  ( X^{(\e)}_1- X_1) ^2 \Big] \leq C\,  c(\e)  ^{1/2} \,.
 \end{equation} 
We point out that  $\| j^{(\e)}_e - j_e\| _\infty \leq \sup _\eta \sum _y  |y| \,|\hat r_\e (y, \eta)|\leq  \beta( \e) $.   Note that $c(\e) \leq \b(\e)$. Hence,  given  $t \in [0,1]$, using \eqref{federico} we get
\begin{multline}\label{ballare}
 \left| S_\e (t) j^{(\e)}_e (\eta) - S(t)j_e(\eta)\right |
    =\left|
  \cE_\eta \left[ j^{(\e)}_e (\eta^{(\e)}_t) -j_e(\eta_t) \right] \right|\\
  \leq \beta(\e) \cP_\eta\bigl( \eta^{(\e)}_t  = \eta_t\bigr)+ 
 (\|j^{(\e)}_e \|_\infty + \|j_e\|_\infty)   \cP_\eta\bigl( \eta^{(\e)}_t \not = \eta_t \bigr)\leq C\, \b (\e)\,.
\end{multline}
In particular, by \eqref{martina} and \eqref{ballare},  the l.h.s. of \eqref{claimA} is bounded by $C \beta(\e)  + C c(\e)^{1/2}$. This concludes the proof of \eqref{claimA}.

\medskip
In order to get \eqref{claimB} we  abbreviate
\begin{equation*}
b^{(\e)}_t:=S_\e(t)j^{(\e)}_e\left(\eta^{(\e)}_1\right)\,, \qquad 
b_{t}:=S(t)j_{e}\bigl(\eta_1\bigr)\,.
\end{equation*}
Then $B^{(\e)}_{T-n}= \int_0 ^{T-n}  ( b^{(\e)}_t - \bbE^{(\e)}_{\eta} (b^{(\e)}_t) )dt$ and  $B_{T-n}= \int_0 ^{T-n}  ( b_t - \bbE_{\eta} (b_t) )dt$. In particular we can bound
 \begin{equation}\label{ginnastica} 
 \mathcal{E}_{\mu_\e}\left[\left(B^{(\e)}_{T-n}-B_{T-n}\right)^2\right]\leq c\sum_{i=1}^{T-n}i^2\int_{i-1}^idt  \mathcal{E}_{\mu_\e}\left[\left(b^{(\e)}_t-\bbE^{(\e)}_{\eta}\left[b^{(\e)}_t\right]-\bigl(b_t-\bbE_{\eta}\bigl[b_t\bigr]\bigr)\right)^2\right]\,.
 \end{equation}
At this point, to get \eqref{claimB} it is enough to show that there exists a  constant $w(\e)$ going to zero as $\beta(\e)$ goes to zero   such that 
\begin{equation}\label{autunno}
 \mathcal{E}_{\mu_\e}\left[\left(b^{(\e)}_t-\bbE^{(\e)}_{\eta}\left[b^{(\e)}_t\right]-\bigl(b_t-\bbE_{\eta}\bigl[b_t\bigr]\bigr)\right)^2\right]\leq  \frac{w(\e)}{({\g-\e})^{3/2}}
  e^{-\frac{\g-\e}{2}t} \,, \qquad \forall t \geq 0\,. 
 \end{equation}
 Since given any $a,b \geq 0$ it holds $\min(a,b) \leq \sqrt{a b}$, it is enough  to show that  the l.h.s. of \eqref{autunno} is bounded from above both by $ w^2(\e)/C$  and by $  C e^{-(\g-\e)t}/(\g-\e)^3$. We start with the latter.

\subsection{The l.h.s. of  \eqref{autunno} is bounded from above by $  C e^{-(\g-\e)t}/(\g-\e)$}
  We observe that 
 \begin{equation}\label{spiedino}
 \begin{split}
&  \mathcal{E}_{\mu_\e}\Big[\Big(b^{(\e)}_t-\bbE_{\eta}^{(\e)}\bigl[b^{(\e)}_t\bigr]\Big)^2\Big]
 \leq  2\mathcal{E}_{\mu_\e}\Big[\Big(b^{(\e)}_t-\mu_\e(j^{(\e)}_e)\Big)^2\Big]+2\mathcal{E}_{\mu_\e}\Big[\Big(\mu_\e(j^{(\e)}_e)-\bbE_{\eta} ^{(\e)}\bigl[b^{(\e)}_t\bigr]\Big)^2\Big]  \\
 & \qquad =2\mu_\e\Big(\Big(S_\e(t)j^{(\e)}_e-\mu_\e(j^{(\e)}_e)\Big)^2\Big)+2\mu_\e\Big(\Big(S_\e(t+1)j^{(\e)}_e-\mu_\e(j^{(\e)}_e)\Big)^2\Big) \\
 & \qquad \leq \verde{4\|j^{(\e)}_e-\mu_\e(j^{(\e)}_e)\|_\infty^2\bigl(\frac{\g}{\g-\e}\bigr)^3e^{-(\gamma-\epsilon)t}},
 \end{split}
 \end{equation}
where  the equality follows from the semigroup property implying that  $\bbE^{(\e)}_\eta(b^{(\e)}_t) = S_\e (t+1) (\eta)$ and  from 
the invariance of $\mu_\e$  for the environment viewed by the perturbed walker.
Moreover, the last inequality follows from \eqref{arachidi_bis}. 
 
 \smallskip
 
 On the other hand, we have
 \begin{eqnarray}
 \mathcal{E}_{\mu_\e}\left[\left(b_t-\bbE_{\eta}\left[b_t\right]\right)^2\right]\leq 2\|j_e\|_\infty \bbE_{\mu_\e}\left[\left|b_t-\bbE_{\eta}\left[b_t\right]\right|\right]= 2\|j_e\|_\infty  \mu_\e(f)\,,
\end{eqnarray} 
where $f(\eta):= \bbE_\eta\bigl[\bigl|b_t-\bbE_{\eta}\bigl[b_t\bigr]\bigr|\bigr]$.
Now, thanks to  \eqref{benedicte}, we can bound
\[ \mu_\e(f)  \leq \mu(f)+ \frac{\e}{\g-\e} \mu (f^2) ^\frac{1}{2} 
\leq \frac{ \g}{\g-\e}  \mu (f^2) ^\frac{1}{2} 
 \leq    \frac{\gamma}{\gamma-\epsilon}\mathcal{E}_{\mu}\Big[\Big(b_t-\bbE_{\eta}\Big[b_t\Big]\Big)^2\Big]^{1/2}\,.
\]
In particular, we conclude that 
\begin{eqnarray}\label{ova}
\mathcal{E}_{\mu_\e}\left[\left(b_t-\bbE_{\eta}\left[b_t\right]\right)^2\right]\leq  \frac{C}{\gamma-\epsilon} \mathcal{E}_{\mu}\left[\left(b_t-\bbE_{\eta}\left[b_t\right]\right)^2\right]^{1/2}.
\end{eqnarray}
Reasoning as in \eqref{spiedino} \verde{(now using directly \eqref{banana} instead of \eqref{arachidi_bis})} we get  that the square of the last factor in \eqref{ova} is bounded by $4\|j_e\|^2e^{-2\gamma t}$. In particular, \eqref{ova} can be refined to
\begin{eqnarray}\label{ovone}
\mathcal{E}_{\mu_\e}\left[\left(b_t-\bbE_{\eta}\left[b_t\right]\right)^2\right]\leq   \frac{C}{\gamma-\epsilon}   e^{-\gamma t}\,.
 \end{eqnarray}
As a byproduct of \eqref{ginnastica}, \eqref{spiedino} and \eqref{ovone} we get that the l.h.s. of  \eqref{autunno} is bounded from above by $  C e^{-(\g-\e)t}/(\g-\e)$.

\subsection{The l.h.s. of  \eqref{autunno} is bounded from above by $ o(1)$} We say that a quantity is $o(1)$ if it goes to zero as $\beta(\e)$ goes to zero.
 Let us write
\begin{equation}\label{mandarino}
\begin{split}
b^{(\e)}_t-b_{t}& = \left [ S_\e(t)j^{(\e)}_e\left(\eta^{(\e)}_1\right)-S(t)j^{(\e)}_{e}\left(\eta^{(\e)}_1\right)\right] +\left[  S(t)j^{(\e)}_e\left(\eta^{(\e)}_1\right)-S(t)j_e \left(\eta^{(\e)}_1\right) \right]\\
&+
\left[ S(t)j_e\left(\eta^{(\e)}_1\right)-S(t)j_e\left({\eta}_1\right)\right]
\end{split}
\end{equation}

Let us  deal with the first term in the r.h.s. We can bound
\verde{\begin{equation*}\label{nebbia}
\begin{split}
\cE_{\mu_\e} &\Big[\Big(S_\e(t)j^{(\e)}_e\big(\eta^{(\e)}_1\big)  -S(t)j^{(\e)}_{e}\big(\eta^{(\e)}_1\big)\Big)^2\Big]
 =\mu_\e\left(\left(S_\e(t)j^{(\e)}_e-S(t)j^{(\e)}_{e}\right)^2\right)\\
& \leq \|j^{(\e)}_e\|_\infty 
\mu_\e\Big( \,|S_\e(t)j^{(\e)}_e- S^{(0)}_\e(t)j^{(\e)}_e |\, \Big)\\ & 
 \leq C  \mu\Big( \,|  S_\e(t)j^{(\e)}_e-S^{(0)}_\e(t)j^{(\e)}_e |\, \Big) + C \e (\g-\e)^{-1} \|  S_\e(t)j^{(\e)}_e-S^{(0)}_\e(t)j^{(\e)}_e\|_\mu \\
 & \leq C' (\g-\e)^{-1}\|  S_\e(t)j^{(\e)}_e-S^{(0)}_\e(t)j^{(\e)}_e\|_\mu\leq C'' \e (\g-\e)^{-1}=o(1)\,.
\end{split}
\end{equation*}}
Indeed, the first identity follows from the  invariance of $(\eta^{(\e)}_t)_{t\geq 0}$ under $\mu_\e$, the second inequality follows from \eqref{benedicte} \verde{and 
\eqref{Sn}}, the third one from Schwarz inequality and the last one from \eqref{volare}with $k=1$. 

\smallskip

We move to the second term which is bounded in uniform norm by
$\| S(t)(j^{(\e)}_e- j_e) \|_\infty \leq \|j^{(\e)}_e- j_e \|_\infty\leq \b (\e)=o(1)$.
On the other hand, using that $\| S(t)j_e \|_\infty$ is uniformly bounded in $t$ and using \eqref{federico}, the $\cE_{  \mu_\e}$--second moment of the third term in the r.h.s. of \eqref{mandarino} can be estimated by $C \cP_\eta ( \eta_1 \not = \eta)\leq C\,c(\e)=o(1)$. 

\smallskip
As a byproduct of the above observations we conclude that
$\cE_{\mu_\e} \left[ \bigl(
b^{(\e)}_t-b_{t}\bigr)^2\right ]=o(1)$.
This also implies that 
\begin{equation*}
\begin{split}
\cE_{\mu_\e}
 \left[
\left( \bbE^{(\e)}_{\eta}\left[b^{(\e)}_t\right]- \bbE_{\eta}\bigl[b_t\bigr]\right)^2\right]& =\cE _{\mu_\e} \left[ \left( \cE_{\eta} [ b^{(\e)}_t - b_t] \right)^2 \right]\\
& \leq \cE _{\mu_\e} \left[   \cE_{\eta}\left [
 (b^{(\e)}_t - b_t)^2 \right]  \right]= \cE _{\mu_\e} \left[    
 (b^{(\e)}_t - b_t)^2   \right] =o(1)\,.
 \end{split}
\end{equation*}
By Schwarz inequality we then conclude that the l.h.s. of  \eqref{autunno} is bounded from above by $ o(1)$.

\subsection{Proof of Claim \ref{Claim2} }\label{dim_claim2}
Let $f_{T,n}(\eta):={\cE_{\eta}}\Big[\bigl(A_1+B_{T-n}\bigr)^2\Big]$. Then \eqref{AtoA0} 
reads
 $\mu_\e( f_{T,n} ) \geq \mu(f_{T,n} ) - C \e /(\g-\e)$. This follows from 
  \eqref{benedicte} if we prove that $\mu ( f_{T,n} ^2)$  is bounded from above uniformly in $T,n$. By Schwarz inequality, it is enough to bound from above ${\cE_{\mu}}\bigl[A_1 ^4\bigr]$ and ${\cE_{\mu}}\bigl[B_{T-n}^4\bigr]$ uniformly in $T,n$. This follows from Lemma \ref{spina}. 
   
 \subsection{Proof of Claim~\ref{Claim3}}\label{dim_claim3}   By standard techniques \cite{spohn,DFGW} we have the following variational characterization of  the diffusion coefficient of a symmetric walker in reversible environment:
\begin{eqnarray}
\langle e, D_0e \rangle =\frac{1}{2}\inf_f\Big\{-2\mu\left(f L_{\rm env}f\right)+\sum_{y\in \bbZ^d} \mu\left(r(y,\eta)\left[ y\cdot e+f(\tau_y\eta)-f(\eta)\right]^2\right)\Big\},\label{eq:varformula}
\end{eqnarray}
where the infimum is taken over local functions $f$ on $\Omega$ and where $e$ is any vector of the canonical basis.


In \eqref{eq:varformula}, by definition of the spectral gap, the first term is bounded from  below by $2\gamma {\rm Var}_\mu(f)$. On the other hand, using the inequality $(a+b)^2\geq \beta a^2-\frac{\beta}{1-\beta}b^2$ for $\beta<1$, we get
\begin{eqnarray*}
\mu\left(r(y,\eta)\left[ y\cdot e+f(\t_y \eta )-f(\eta)\right]^2\right)\qquad\qquad\qquad\qquad\qquad\qquad\\
\qquad\qquad\geq \beta \mu(r(y,\cdot))(y\cdot e)^2-\frac{\beta}{1-\beta}\mu\left(r(y,\eta)\left[f(\tau_y\eta)-f(\eta)\right]^2\right)\\
\qquad\qquad\qquad\qquad\geq \beta \mu(r(y,\cdot))(y\cdot e)^2-4\,\underset{\eta}{\sup}\, r(y,\eta)\frac{\beta}{1-\beta}Var_\mu(f).
\end{eqnarray*}

Injecting this in \eqref{eq:varformula} and choosing $\beta<1$ so that 
\[
2\gamma-4\frac{\beta}{1-\beta}\sum_{y\in\bbZ^d}\,\underset{\eta}{\sup}\, r(y,\eta) \leq 0,
\]
we get $\langle e, D_0e \rangle >0$. Hence, we conclude that (cf.\@ \cite[Eq. (2.43)]{DFGW})
\begin{equation}\label{Dzero}
\underset{T\to +\infty}{\lim}\frac{1}{T} {\cE_\mu}\Big[\bigl(X_T\cdot e\bigr)^2\Big]=\langle e, D_0e \rangle>0.
\end{equation}
\medskip

We claim that 
\begin{equation}\label{tortina}
\sup_{T\geq 0} 
\mu\left({\cE_{\eta}}\left[X_T\cdot e \right]^2\right)<+\infty
\end{equation}
For simplicity we  restrict the proof to $T$ integer (indeed, to our final aim this would be enough, anyway 
one could  extend the thesis to the general case). Due to the Markov property, we get 
\begin{equation}\label{40} 
{\cE_{\eta}}\Big[X_T\cdot e \Big]  =
{\cE_{\eta}} \Big[  \sum_{k=0}^{T-1}   ( X_{k+1} \cdot e- X_{k}\cdot e) \Big]  = 
{\cE_\eta}  \Big[  \sum_{k=0}^{T-1}{  \cE_{\eta_k}}   ( X_{1}\cdot e ) \Big] = \sum_{k=0}^{T-1}{ \cE_\eta} \Big[{ \cE_{\eta_k}}   ( X_{1}\cdot e ) \Big]\,.
\end{equation}
Consider now the function {$f(\eta)=\cE_\eta   ( X_{1}\cdot e )$.}

 Since $S(t) f(\eta)= { \cE_\eta \Big[  \cE_{\eta_t}   ( X_{1}\cdot e ) \Big]}$, from \eqref{40} we get that
\[ {\cE_{\eta}}\Big[X_T\cdot e \Big] =  \sum_{k=0}^{T-1} S(k) f (\eta)\,.\]
Note that $\mu(f)=0$  by reversibility and that $f \in L^2 (\mu)$  by Lemma \ref{genova} adapted to the unperturbed process and by condition \eqref{nthMoment}. By the Poincar\'e inequality \eqref{banana} we conclude that
$\| S(t) f\| \leq e^{-\g t} \|f\| $.
At this point we have
\begin{equation}
\begin{split}
\mu\Big(\bbE_{\eta}\Big[X_T\cdot e \Big]^2\Big)&= \mu \Big[ \Big(  \sum_{k=0}^{T-1} S(k) f (\eta)
\Big)^2 \Big]= \|    \sum_{k=0}^{T-1} S(k) f \|^2 \leq \Big(       \sum_{k=0}^{T-1} \| S(k) f \| \Big)^2\\
& \leq  \|f\|^2 \Big(       \sum_{k=0}^{T-1} e^{-\g k }  \Big)^2 \leq \frac{ \|f\|^2  }{1-e^{-\g} }\,,
\end{split}
\end{equation}
thus concluding the proof of  \eqref{tortina}. Trivially, Claim \ref{Claim3} follows as a byproduct of \eqref{Dzero} and \eqref{tortina}. 

\appendix
\section{Miscellanea}\label{app_misc}

{
Lemma \ref{miele} and  Lemma \ref{SC}   below have a standard derivation and therefore we omit their proof. Detailed proofs can be found in \cite[Appendix A]{ABFv2}.}
\begin{Lemma}\label{miele} Let $\O$ be  a metric space  and let $\nu$ be   a Borel probability measure on $\O$. Then:
\begin{itemize}
\item[(i)] The subset $C_b(\O)$ of  bounded continuous functions $f: \O\to \bbR$ is dense in $L^2(\nu)$.
\item[(ii)] Let $h$ be a function in $L^2(\nu)$ such that  $\nu(h f) \geq 0$ for any $f \in C_{b,+}(\O):=\{ g \in C_{b}(\O): g\geq 0\}$.
Then, $h \geq 0$ $\nu$--a.s..
\end{itemize}
\end{Lemma}

\begin{Lemma}\label{SC} The semigroup $S(t)$, $t \in \bbR_+$, defined at the beginning of Section \ref{viva_segovia} is strongly continuous.
\end{Lemma}

\begin{Lemma}\label{equitalia} In the same setting of Section \ref{viva_segovia}, 
given a positive constant $\g>0$, \eqref{poincare} is equivalent to \eqref{banana}.
\end{Lemma}
The above lemma is usually proven in the reversible case. We give the proof to stress that it holds even without reversibility. 
\begin{proof} 
For any $f \in \cD(L) $ the map $[0,+\infty) \ni t \to S(t) f \in L^2 (\mu)$ is $C^1$, $S(t) f \in \cD (L)$ and  $ \frac{d}{dt} S(t) f = L S(t) f$ \cite[Chap. II, Sec. 1]{EN}. In particular, taking $f \in \cD(L) $,  by differentiating one gets 
\begin{equation}\label{diffi}
 \frac{d}{dt} \| S(t) f \|^2 =\langle   L S(t) f ,  S(t) f \rangle +\langle   S(t) f , L S(t) f \rangle=
2  \langle   S(t) f , L S(t) f \rangle\,,
\end{equation}
where $\langle \cdot , \cdot \rangle $ denotes the scalar product in $L^2(\mu)$ (note that we have used the symmetry of the scalar product: $\langle g,g'\rangle= \langle g', g\rangle$).

We first assume   Poincar\'e inequality \eqref{poincare} to be satisfied and take  $f \in \cD(L) $  with $\mu(f)=0$. By \eqref{diffi} 
and  the Poincar\'e inequality, one gets 
\[  \frac{d}{dt} \| S(t) f \|^2  = 2  \langle   S(t) f , L S(t) f \rangle \leq -2 \g  \| S(t) f \|^2  \,.\]
 Note that we have used the stationarity of $\mu$, implying that $\mu(S(t)f)= \mu(f)=0$.  Gronwall inequality then leads to $ \| S(t) f \| \leq e^{-\l t} \|f\|$.  In particular, \eqref{banana} holds for any  $f \in \cD(L)$ with $\mu(f)=0$, and therefore for any $f \in \cD(L)$ (observe  that constant functions are left invariant by  $S(t)$). By density of $\cD(L)$ in $L^2(\mu)$ one gets \eqref{banana} for any $f \in L^2(\mu)$.

We now assume \eqref{banana} to be satisfied and fix  $f \in \cD(L) $  with $\mu(f)=0$.  By \eqref{diffi} we have $\| S(t)f \|^2 = \|f\|^2 -  2 t  \langle   f ,- L  f \rangle+o(t)$ as $t \downarrow 0$. On the other hand, $e^{- 2 \g t} \| f \|^2 = \|f\|^2 - 2 \g \|f\|^2 +o(t)$ as $t \downarrow 0$. Hence the Taylor expansion of \eqref{banana} implies \eqref{poincare}.
\end{proof}
The following lemma extends the probabilistic interpretation of the semigroup $S_\e(t)$ given in \eqref{ghiacciolo}.
\begin{Lemma}\label{candelina}
Consider the same assumptions of Theorem \ref{teo_invariante}. Then, given $f \in L^2(\mu)$, it holds
$ S_\e (t) f (\eta)= \bbE^{(\e)}_\eta \bigl( f(\eta_t) \bigr) $   $ \mu_\e\text{--a.s.}$\end{Lemma}
\begin{proof}
By Lemma \ref{miele} there exists a sequence $(f_n)_{n \geq 1}$ in $C_b(\O)$ with $\| f_n - f\| \to 0$ as $n \to \infty$. Since $S_\e(t)$ is a bounded operator in $L^2(\mu)$, we get that    $\| S_\e(t) f_n - S_\e(t) f\|\to 0$ as $n \to \infty$. In particular, at cost to extract a subsequence, we have $S_\e (t) f_n (\eta) \to S_\e (t) f (\eta)$ for $\mu$--a.e.\@ $\eta$. Since $\mu_\e \ll \mu$ (by Theorem \ref{teo_invariante}), we conclude that  
\begin{equation}\label{lirica1}
S_\e (t) f_n (\eta) \to S_\e (t) f (\eta) \text{  for $\mu_\e$--a.e.\@ $\eta$}\,.
\end{equation}
On the other hand, by the stationarity of $\mu_\e$ for the perturbed dynamics, we have
\begin{multline*}
\mu_\e \Big[\Big| \bbE^{(\e)}_\eta \bigl( f_n(\eta_t) \bigr)-\bbE^{(\e)}_\eta \bigl( f(\eta_t) \bigr)
\Big|\Big]  \leq \mu_\e  \left[  \bbE^{(\e)}_\eta \left(\left| f_n(\eta_t) -f(\eta_t) \right|\right)\right] =
\bbE^{(\e)} _{\mu_\e} [|f_n(\eta_t)-f(\eta_t)| ] \\
= \mu_\e [|f_n-f|)= \mu \left[\frac{d \mu_\e}{d \mu} |f_n-f|\right] \leq 
\| \frac{d \mu_\e}{d \mu} \| \cdot \| f_n -f\| \to 0 \,. 
\end{multline*}
We have shown that the map $\bbE^{(\e)}_\cdot \bigl( f_n(\eta_t) \bigr)$ converges to the map
$\bbE^{(\e)}_\cdot \bigl( f(\eta_t) \bigr)$ in $L^1(\mu_\e)$. Hence, at cost to extract a subsequence, 
the convergence holds also $\mu_\e$--a.s.. The thesis is then a byproduct of the last observation, of \eqref{lirica1} and the identity \eqref{ghiacciolo} applied to $f_n\in C_b (\O)$ instead of $f$ (which holds $\mu$--a.s.\@ and therefore $\mu_\e$--a.s.\@ since $\mu_\e \ll \mu$).
\end{proof}
 
%

\bigskip

{\bf Acknowledgements}. The authors  thank L. Bertini, M. Mariani, S. Olla and P. Mathieu for useful discussions. O. Blondel and L. Avena acknowledge the Department of Mathematics of the University La Sapienza in Rome for the kind hospitality. O. Blondel and A. Faggionato thank the organizers of the Workshop  ``Random Motion in Random Media"  (Eurandom, The Netherlands), during which part of this work has been completed.
{L. Avena has been supported by NWO Gravitation Grant 024.002.003-NETWORKS.}


\end{document}